\documentclass[12pt]{article}
\usepackage{amsmath,amssymb, latexsym}
\usepackage{color}
\def\seq#1#2#3{#1_{#2},\,\ldots,#1_{#3}}
\def\useq#1#2#3{#1^{#2},\,\ldots,#1^{#3}}
\def\val#1#2#3{(#1_1(#3),\,\ldots,#1_#2(#3))}
\def\spec{\text{\tt Spec\,}}
\def\init{\text{\tt in\,}}

\def\union{\textstyle\bigcup\limits}

\def\tq{\;|\;}
\def\ideal#1{\text{\tt #1}}
\def\im{\text{\tt Im\,}}

\def\suma{\textstyle\sum\limits}
\def\pr{\textstyle\prod\limits}

\def\vv{{\underline{v}}}
\def\nuv{{\underline{\nu}}}
\def\t{{\underline{t}}}

\def\mm{{\underline{m}}}
\def\nn{{\underline{n}}}
\def\kk{{\underline{k}}}
\def\ee{{\underline{e}}}
\def\k{\overline}

\def\w{\widetilde}
\def\betab{\bar\beta}
\def\PP{\mathbb P}

\def\ZZ{\mathbb Z}
\def\C{\mathbb C}

\def\1{{\underline{1}}}
\def\gr{{\rm gr}}

\newtheorem{theorem}{Theorem}
\newtheorem{corollary}{Corollary}
\newtheorem{lemma}{Lemma}
\newtheorem{proposition}{Proposition}

\newenvironment{remark}
{\smallskip\noindent{\bf Remark\/}.}{\smallskip\par}

\newenvironment{proof}
{\noindent{\bf Proof\/}.}{{ $\Box$}\smallskip\par}
\newenvironment{Proof}
{\noindent{\bf Proof\/}.}{{}\smallskip\par}

\title{Generating sequences and Poincar\'e series
for a finite set of plane divisorial valuations
\footnote{\noindent Math. Subject Class. 14B05, 16W50, 16W70,
13A18
\newline
Supported by  the Spain Ministry of Education
MTM2004-00958 and JCyL VA025A07. Second author also supported
by
Bancaixa P1-1A2005-08.\newline The authors would like to thank the referee for his/her  detailed and 
helpful comments.\newline {\bf Addresses:} \newline F. Delgado \& A.
N\'u\~nez: Dto. de \'Algebra, Geometr\'{\i}a y Topolog\'{\i}a.
Universidad de Valladolid. 47005 Valladolid. Spain. e-mail:
fdelgado@agt.uva.es, anunez@agt.uva.es
\newline
C. Galindo: Dto. de  Matem\'{a}ticas (ESTCE), UJI, Campus Riu Sec.
12071  Castell\'{o}n. Spain. e-mail: galindo\symbol{'100}mat.uji.es
} }
\author{F. Delgado
\and C. Galindo \and A. N\'u\~nez }
\date{}

\begin{document}
\maketitle
\begin{abstract}
Let $V$ be a finite set of divisorial valuations centered at a
2-dimensional regular local ring $R$. In this paper we study its
structure by means of the semigroup of values, $S_V$, and the
multi-index graded algebra defined by $V$, $\gr_V R$. We prove
that $S_V$ is finitely generated and we compute its minimal set of
generators following the study of reduced curve singularities.
Moreover, we prove a unique decomposition theorem for the elements
of the semigroup.
 The comparison between valuations in $V$, the
approximation of a reduced plane curve singularity $C$ by families
of sets $V^{(k)}$ of divisorial valuations, and the relationship
between the value semigroup of $C$ and the semigroups of the sets
$V^{(k)}$, allow us to obtain the (finite) minimal generating
sequences for $C$ as well as for $V$.

We also analyze the structure of the homogeneous components of $\gr_V R$. The study of
their dimensions allows us to relate the Poincar\'e series for $V$ and for a general
curve $C$ of $V$. Since the last series coincides with the Alexander polynomial of the
singularity, we can deduce a formula of A'Campo type for the Poincar\'e series of $V$.
Moreover, the Poincar\'e series of $C$ could be seen as the limit of the series of
$V^{(k)}$, $k\ge 0$.
\end{abstract}

\section*{Introduction}

This paper deals with the structure of a finite number of divisorial
valuations centered at a regular local ring of dimension two. In
singularity theory there are many problems that involve finitely
many interrelated exceptional divisors (and so, their corresponding
divisorial valuations), which cannot be analyzed independently
without losing some information. Classification of sandwiched
singularities, minimal resolutions and the Nash problem are examples
of this situation. The study of plane curve singularities
constitutes a similar situation and the treatment of a branch is
rather different of the one of the whole curve (see \cite{g-d-c} and
\cite{c-d-g}). Problems as uniformization and monomialization of
valuations, studied historically by Zariski and Abhyankar, are also
object of recent activity (see e.g. \cite{t}),  providing another
motivation to our study.

This paper is inspired in two sources. Firstly,  the results for the
case of a single divisorial valuation by Spivakovsky \cite{s-2},
 where minimal generating sequences were computed (see also
\cite{gre-kiy}), and Galindo, who computes the Poincar\'e series
\cite{g-2}. The second one is the set of results \cite{lon, g-d-c,
c-d-g} obtained by Campillo, Delgado and Gusein-Zade for a
plane curve singularity with several branches where generation of
the semigroup, zeta function and Poincar\'e polynomial are
considered.
\medskip

Throughout this paper, we will assume that $(R,\ideal{m})$ is a
local, regular, complete and $2$-dimensional ring and that it has
an algebraically closed coefficient field. Replacing curves defined by elements in $R$ by analytically reduced curves
defined by elements in the completion $\hat R$ of $R$, and considering the valuations defined by their branches in the
ring $R$ (see e.g. \cite{zeit}),  theorems stated in the paper remain true without the assumption of
completeness. However, we will consider the complete case because it simplifies the
proofs and gives a more clear intuition.

\smallskip

Each irreducible component $E_\alpha$ of the exceptional divisor $E$ of a modification
$\pi$ of $\spec R$ defines a valuation of the fraction field of $R$ centered at $R$,
named divisorial and denoted by $\nu_{\alpha}$. An irreducible element in $R$ such that
the strict transform by $\pi$ of the corresponding curve is smooth and intersect
transversely $E_{\alpha}$ at a smooth point of $E$ is, generically, denoted by $Q_\alpha$
and plays an important role in this paper. Consider a finite set $V=\{\seq \nu 1r\}$ of
$r \geq 1$ divisorial valuations associated to exceptional components $E_{\alpha(i)}$ of
$E$ where $\pi : (X,E)\to (\spec R, \ideal{m})$ is the minimal modification such that
$E_{\alpha(i)} \subset E$ for $1 \le i \le r$. We define the semigroup of values of $V$
as the subsemigroup of $\ZZ_{\ge 0}^r$ given by  $S_V := \{\nuv(f):=(\nu_1(f), \ldots,
\nu_r(f))| f \in R\setminus\{0\}\}$. A general curve of $V$ is a reduced plane curve with
$r$ branches each one defined by an equation $Q_{\alpha(i)} = 0$. When $r=1$, $S_V$
coincides with the semigroup of values of any  general curve of $V$ \cite{tei}. The
valuation ideals $J(\mm) = \{g\in R \, | \, \nuv(g)\ge \mm\}$ define a multi-index
filtration of the ring $R$ which gives rise to a graded algebra $\gr_V R =
\bigoplus\limits_{\mm\in \ZZ_{\ge 0}^r} J(\mm)/J(\mm+\ee)$, $\ee = (1,\ldots,1)$. This
paper analyzes both objects, the semigroup and the graded algebra, for a set valuations
$V$ looking for its essential arithmetical and algebraic properties.

\smallskip

A description of the semigroup $S_V$  is given in Section 2. In
Theorem~\ref{semi}, we give the minimal set of generators of
$S_V$, proving that it is finitely generated, unlike the case of a
reduced plane curve singularity (see \cite{man} and \cite{lon}).
The set $\{B^i:= \nuv(Q_{\alpha(i)}) \mid i=1,\ldots,r\}$ plays a
very special role in $S_V$. Proposition~\ref{cor10} shows that
the projectivization of the vector space $D(B^i) =
J(B^i)/J(B^i+\ee)$ is canonically isomorphic to the exceptional
divisor $E_{\alpha(i)}$ which defines the valuation $\nu_i$. In
particular, $D(B^i)$ is bidimensional. The study of the dimension
$d_i(\mm)$ of the spaces $D_i(\mm)=J(\mm)/J(\mm+e_i)$ allows to
prove Theorem~\ref{str_semi}, which gives a unique decomposition
for the elements in $S_V$ in terms of the set $\{\useq{B}1r\}$. In
particular, we give another proof of the fact that if  $V$ consists of all the divisors of a modification,
then $S_V$ is a free semigroup generated by $\useq{B}1r$.

\smallskip

In Section 3 we describe a generating sequence for a finite set
$V$ of divisorial valuations  and for a reduced curve with several
branches. Denote by ${\cal E}$ the set of end divisors of the
minimal resolution of $V$, i.e., the exceptional components
$E_\alpha$ such that $E\setminus E_\alpha$ is connected, and set
$\Lambda_{\cal E} = \{Q_{\rho}\tq E_\rho\in {\cal E}\}$. The main
result of this section, Theorem~\ref{maint}, states that
$\Lambda_{\cal E}$ is a minimal generating sequence of $V$, that
is, any valuation ideal is generated by monomials in the set
$\Lambda_{\cal E}$. After a result of Campillo and Galindo
\cite{c-g}, this is equivalent to the fact that $\gr_V R$ is the
$R/\ideal{m}$-algebra generated by the classes of the elements in
$\Lambda_{\cal E}$.

A similar result is true for the set $W$ of valuations defined by the branches of a
reduced plane curve $C$, however we must change the set $\Lambda_{\cal E}$ by another one
$\Lambda_{\k{\cal E}}$ which is also finite (see, again, Theorem~\ref{maint}). The key to
understand it, is that $W$ can be regarded as a limit of families of divisorial
valuations $V^{(k)}$, $k \geq 0$, and so  $\Lambda_{\k{\cal E}}$ as the limit
  of the sequence $\Lambda_{{\cal
E}^{(k)}}$ given by $V^{(k)}$. The number of classes in $\gr_V R$
produced by each element in $\Lambda_{\cal E}$ is finite, however
some elements in  $\Lambda_{\k{\cal E}}$  give infinitely many
different classes
in the corresponding algebra (see the last remark of Section
\ref{gradedalg}). This fact explains the apparent contradiction
between the infinite generation of the semigroup of a plane curve
singularity and the existence of a finite generating sequence.

It is worthwhile to mention that the so called multipliers ideals of ideals in the ring
$R$, can be regarded as ideals $J(\mm)$ for concrete sets, $V$, and  elements $\mm$
\cite{laz}. Notice that, in our case, these ideals are exactly the complete ones
\cite{lip-wat}.

\smallskip

The dimensions $d(\mm) = \dim J(\mm)/J(\mm+\ee)$ of the homogeneous pieces of the graded
ring $\gr_V R$ can be collected in the Laurent series $L_V(\seq t1r) = \suma_{\mm\in
\ZZ^r} d(\mm) \t^{\mm}$
(note that the sum extends to $\ZZ^r$). Following \cite{c-d-g} and \cite{d-g}, the
Poincar\'e series of $V$ is defined as the formal series with integral coefficients
$$
P_V(\seq t1r) = \frac{L_V(\seq t1r)\cdot\pr_{i=1}^r(t_i-1)}{t_1 t_2
\cdots t_r -1} \; .
$$
As an application of the results and  techniques developed in the previous sections,
Section 4 is devoted to the computation of the Poincar\'e series $P_V$. So, in
Theorem~\ref{poincare1} we state the relation between the Poincar\'e series of $V$ and
the Poincar\'e polynomial $P_C$ of any general curve $C$ of $V$. This polynomial
coincides with the Alexander polynomial of the link of the singularity \cite{c-d-g}.
In the complex case, the above expression leads to an explicit formula for $P_V$ in terms
of the topology of the exceptional divisor, very similar to the formula of A'Campo (see
\cite{acam}) for the zeta function (extended by Eisenbud and Neumann in \cite{EN} for the
Alexander polynomial):
$$
P_V(\seq t1r) = \pr\limits_{E_\alpha\subset{E}}
\left(1-\t^{{\nuv}^\alpha}\right)^{-\chi({\stackrel{\bullet}{E}}_\alpha)}
$$
where
$\chi({\stackrel{\bullet}{E}}_\alpha)$ is the Euler
characteristic  of the smooth part
${\stackrel{\bullet}{E}}_\alpha$ of $E_\alpha\subset E$.
This  formula was
conjectured by the authors some time ago, but the first complete
proof has been given in \cite{d-g}
by using a very different approach: the integration on
infinite dimensional spaces with respect to the Euler
characteristic.

\smallskip

To prove our results, we develop two different kind of techniques
which in our opinion have interest by themselves. The main steps
of the first one are included in Section~1. There, we consider
pairs of elements in $R$ with the same value by a valuation
$\nu_\alpha$ associated to a component $E_\alpha$ of the
exceptional divisor $E$ of a modification and we find the relation
between the initial forms of these elements with respect to
$\nu_\alpha$ as well as their values for the valuations
corresponding to other components of $E$. Such study is given in
terms of the topology of the exceptional divisor. In the proofs we
systematically use the geometry of pencils of plane curves. In
particular, the fact that the divisor $E_\alpha$ be dicritical for
the pencil $\{\lambda f + \mu g\}$ if and only if the initial
forms of the functions $f$ and $g$ with respect to $\nu_\alpha$
are linearly independent, explains the deep relationship between
both concepts.

The second technique, specially used in Section 4, is the cited
approximation of curves by divisorial valuations. We can see it as
a way to go from  results related to a curve $C$ to  similar
results for the corresponding sets $V^{(k)}$ of divisorial
valuations, and viceversa. Corollary~\ref{cor}, which presents the
Poincar\'e polynomial of a general curve $C$ of $V$ as the limit
of the Poincar\'e series of sets of divisorial valuations
$V^{(k)}$, gives a good example of this philosophy.

\section{Divisorial valuations}\label{prelim}

Let $(R,\ideal{m})$ be a local, regular, complete and $2$-dimensional ring with an
algebraically closed coefficient field $K$. For us, a curve will be a subscheme of
$\spec R$, $C_f$, defined by some element $f \in \ideal{m}$.  A {\bf divisorial
valuation} $\nu$ is a discrete valuation of the fraction field of $R$, centered at $R$
(i.e., $R\cap \ideal{m}_\nu = \ideal{m}$, where $(R_{\nu},\ideal{m}_\nu)$ is the
valuation ring of $\nu$), with rank 1 and transcendence degree $1$.

\medskip

Given a \textbf{modification}, that is, a finite sequence of point
blowing-ups, $\pi:X\to \spec R$, there is a divisorial valuation,
$\nu_{\alpha}$, associated to each irreducible component
$E_{\alpha}$ of the total exceptional divisor $E$ of $\pi$, namely,
for $f\in R$, $\nu_\alpha(f)$ is the vanishing order of the function
$f\circ \pi:X\to K$ along the divisor $E_\alpha$. We will say that
$\nu_{\alpha}$ is the $E_{\alpha}$-valuation.

Assume that $\pi$ is given by the sequence
\[
\pi:   X = X_{N+1}    \stackrel{\pi_{N+1}}{\longrightarrow} X_{N}
\longrightarrow           \cdots            \longrightarrow X_{1}
\stackrel{\pi_{1}}{\longrightarrow} X_{0}=\spec R,
\]
and denote, for $0\le i\le N$, by $P_i$ the center of $\pi_{i+1}$ in $X_i$ ($P_0=\ideal
m$), by $(R_i, \ideal m_i)$ the local ring of $X_i$ at $P_i$, and by $E_{i+1}$ the
exceptional divisor of $\pi_i$.  Then, for $1\le \alpha\le N+1$, $\nu_{\alpha}$ is the
$\ideal m_{\alpha-1}$-adic valuation. Given $\nu=\nu_{\alpha}$, we will sometimes denote
$P_{\nu}$ and $E_{\alpha(\nu)}$ instead of $P_{\alpha}$ and $E_{\alpha}$.

In fact, divisorial valuations correspond $1-1$ to finite sequences of point blowing-ups,
by associating to $\nu$ its {\bf minimal resolution}, defined as follows:  $\pi_{i+1}$ is
the blowing-up of $X_i$ at $P_i$, $P_0=\ideal m$, and for $i\ge 1$ $P_i$ is the unique
point in the exceptional divisor of $\pi_{i}$, $E_i$, such that $R_{\nu}$ dominates the
local ring of $X_i$ at $P_i$. In this way, $\nu$ is the divisorial valuation associated
to $E_{N+1}$.

Given the $E_{\nu}$ divisorial valuation $\nu$, denote by ${\cal
C}_{\nu}$ the set of all irreducible curves in $\spec R$ whose
strict transform by the minimal resolution $\pi$ of $\nu$ is
smooth and meets $E_{\nu}$ transversely at a nonsingular point of
the total exceptional divisor of $\pi$. An element $f \in
\ideal{m}$ is said to be a {\bf general element} of $\nu$ if
$C_{f} \in {\cal C}_{\nu}$. In \cite{s-2} it is proved that for $f
\in R$,
\begin{equation}\label{eq00}
\begin{array}{rl}
\nu(f) &= \min \{ (f,g)\; |\; g \in {\cal C}_{\nu} \}\, \\
    &= (f,g)\; \text{ if }\; \widetilde{C}_f \cap \widetilde{C}_g
= \emptyset \; \text{ and } g\in {\cal C}_\nu\; ,
\end{array}
\end{equation}
where $(f,g)$ stands for the intersection multiplicity $(C_f,C_g)$
between the curves $C_{f}$ and $C_{g}$ and $\widetilde{C}_f,
\widetilde{C}_g$ for the strict transforms by $\pi$ of the curves
$C_f$ and $C_g$. The minimal resolution $\pi: X\to \spec R$ of the
divisorial valuation $\nu$ is an embedded resolution of $C_{f}$
for $f\in {\cal C}_{\nu}$, in general not the minimal one.

Conversely, let $C_f$ be
any irreducible curve in $\spec R$, and take the associated
(infinity) sequence of blowing-ups with centers at the infinitely
near points of $f$,
\begin{equation}\label{seq}
\cdots\longrightarrow      X_{i+1}
\stackrel{\pi_{i+1}}{\longrightarrow} X_{i} \longrightarrow \cdots
\longrightarrow X_{1} \stackrel{\pi_{1}}{\longrightarrow}
X_{0}=\spec R.
\end{equation}
For each $i\ge 0$ set $\nu_i$ the $E_{i+1}$ divisorial valuation.
Since the curve $C_f$ is determined by the sequence (\ref{seq}), we can
think of the sequence of valuations $\{\nu_i\}_{i\ge 0}$ as an
approaching of $C_f$. Indeed, for any $g\in R$, nonzero
in the ring $R/(f)$,
$\nu_i(g)=(f,g)$ for $i\gg 0$. So, the study of divisorial
valuations and of irreducible curves is closely related (see for
example \cite{s-2}).

\medskip

Let $\pi: (X,E)\to (\spec R, \ideal{m})$ be a modification. The
{\bf dual graph} ${\cal G}(\pi)$ of $\pi$ is the dual figure of
the exceptional divisor $E$; that is, it is a graph with a vertex
$\alpha$ for each irreducible component $E_{\alpha}$ of $E$ and
where two vertices are adjacent if and only if their corresponding
exceptional divisors intersect.

The graph ${\mathcal G}(\pi)$ is a tree. We will denote by ${\bf 1}$
the vertex corresponding to the first exceptional divisor and by
$[\beta, \alpha]$ the path joining $\beta$ and $\alpha$. Along this
paper, for a vertex $\alpha$ in ${\mathcal G}(\pi)$, $Q_{\alpha}$
will stand for any irreducible  element of $\ideal m$ such that the
strict transform of the curve $C_{Q_{\alpha}}$ on $X$ is smooth and
meets $E_{\alpha}$ transversely at a nonsingular point.
$C_{Q_{\alpha}}$ gives in particular a general element of the
$E_{\alpha}$-valuation.

A {\bf dead end} (respectively, {\bf  star vertex}) of the graph
${\cal G}(\pi)$ is a vertex which is adjacent to a unique
(respectively, to at least three) vertices.  The
set of dead ends will be
denoted by ${\cal E}$. Given a dead end $\rho\neq {\bf 1}$,
$st_\rho$ will denote the nearest to $\rho$
star vertex of ${\cal G}(\pi)$.

\medskip
In this paper we will make use of the concept of
\textbf{pencil} of elements in $R$.
Recall that if we consider the pencil $L=\{\lambda f+\mu g\tq
\lambda, \mu\in K\}$,
relative to two elements $f,g\in R$, a component $E_{\alpha}$ of the exceptional divisor
$E$ of a modification $\pi$ is said to be {\bf dicritical} for $L$ if the
$E_{\alpha}$-valuation, $\nu_{\alpha}$, is constant on $L$. This
condition is equivalent to say that the lifting
$\w{\varphi}=\varphi\circ\pi$ of the rational function $\varphi=f/g$ to $X$ restricts to
a surjective (that is, non constant) morphism from $E_{\alpha}$ onto $\mathbb P^1_K$. In
the sequel, we will identify $\PP^1_K$ with $K \cup \{ \infty \}$. The fibers of $L$ are
studied in \cite{pencils} in the analytic complex case, and we will use those results
because they can be easily extended to our context.

In particular from Theorems 1, 2 and 3 in \cite{pencils} we can
deduce the following:

Let $\pi:(X,E)\to \spec R$ be a modification and $\alpha$ a vertex
of ${\cal G}(\pi)$. For a subset $A$ of $\mathcal G(\pi)$ denote
$E_A=\union_{\beta\in A}E_{\beta}$. Assume that $\w{\varphi}$ is
constant in $E_{\alpha}$, $\w{\varphi}|_{E_\alpha}\equiv c\in
\PP^1$. Then the strict transform $\w{C}_{f-cg}$ of $C_{f-cg}$
intersects $E_A$, $A$ being the maximal connected subset of ${\cal
G}(\pi)$ such that $\alpha\in A$, and $\w{\varphi} \equiv c$ along
$E_A$.

On the other hand, assume that
$E_\alpha$ is dicritical and $P\in E_\alpha$ is
such that $\w{\varphi}(P)=c$.
If $P$ is a smooth point of
$E$  then $\w{C}_{f-cg}$
intersects
$E_\alpha$ in $P$, and
if $P$ is singular and $\Delta$ is the
connected component of ${\cal G}(\pi)\setminus \{\alpha\}$ such
that
$E_\Delta\cap E_\alpha =\{P\}$ then
$\w{C}_{f-cg}$ intersects $E_\Delta$.

\bigskip

Next results are stated for a modification $\pi:(X,E)\to \spec R$ and a vertex
$\alpha\in{\cal G}(\pi)$.

\begin{lemma}\label{lem12}
Let $h\in R$ be such that $\nu_{\alpha}(h)=\nu_{\alpha}(Q_\alpha)$
and assume that $\w{C}_h\cap\w{C}_{Q_{\alpha}}=\emptyset$. Then
$E_{\alpha}$ is the unique dicritical divisor of the pencil
$L=\{\lambda Q_{\alpha}+\mu h\tq \lambda,\mu\in K\}$, and the
lifting $\w{\varphi}$ of the rational function
$\varphi=Q_{\alpha}/h$ to $X$ restricts to an isomorphism in
$E_{\alpha}$.
\end{lemma}

\begin{proof}
Since $\nu_\alpha(Q_\alpha)=\nu_{\alpha}(h)$, $\w{\varphi}$ is defined in every point of
$E_{\alpha}$, and $\w{\varphi}_{\alpha}:=\w{\varphi}|_{E_{\alpha}}\not\equiv 0,\,
\infty$. Since moreover $\w{C}_h\cap\w{C}_{Q_{\alpha}}=\emptyset$,
$\w{\varphi}_{\alpha}(\w{C}_{Q_{\alpha}}\cap E_{\alpha})=0$, so $E_{\alpha}$ is a
dicritical component for $L$. In fact, $E_{\alpha}$ is the unique dicritical component
for $L$, because the existence of another one would contradict the irreducibility of the
fiber  $Q_{\alpha}$.

Let $P\in E_{\alpha}$ be such that $\w{\varphi}_{\alpha}(P)=0$. If
there were a connected component $\Delta$ of $\mathcal
G(\pi)\setminus\{\alpha\}$ such that $P\in E_{\Delta}$, then
$\w{C}_{Q_{\alpha}}$  would intersect $E_{\Delta}$,
which is impossible by the election of $Q_{\alpha}$.
Hence, $P$ is a smooth point of $E$, $P=\w{C}_{Q_{\alpha}}\cap
E_{\alpha}$. Moreover, from Theorem 3 of \cite{pencils}, $P$ is not
a critical point of $\w{\varphi}_\alpha$, so $\w{\varphi}_{\alpha}$
has degree 1, i.e., it is an isomorphism.
\end{proof}

The next result is a generalization of Lemma 4 in \cite{extended}.

\begin{proposition}\label{lem1}
Let $h\in R$ be such that $\nu_{\alpha}(h)=\nu_{\alpha}(Q_\alpha)$
and such that the strict transform $\w{C}_h$ of $C_h$ on $X$ does
not intersect $E_{\alpha}$. Then there exists a unique connected
component $\Delta$ of  ${\cal G}(\pi)\setminus \{\alpha\}$ such that
$\w{C}_h\cap E_{\Delta}\neq\emptyset$. Moreover,
$\nu_\gamma(h)=\nu_\gamma(Q_\alpha)$ if $\gamma \in {\cal
G}(\pi)\setminus \Delta$, and $\nu_\gamma(h)>\nu_\gamma(Q_\alpha)$
otherwise.
\end{proposition}

\begin{proof}
Keep the notations of Lemma \ref{lem12}. Let $\Delta$ be a
connected component of $\mathcal G(\pi)\setminus\{\alpha\}$ such
that $\w{C}_h\cap E_\Delta \neq \emptyset$. By Lemma \ref{lem12},
there are not dicritical divisors of $L$ in $\Delta$, and since
$\w{\varphi}(P)=\infty$ for any $P\in \w{C}_h\cap E_\Delta$, then
$\w{\varphi}|_{E_{\Delta}}\equiv\infty$, which implies
$\nu_{\gamma}(h)>\nu_{\gamma}(Q_{\alpha})$ for $\gamma\in\Delta$.
As $\w{\varphi}_{\alpha}$ is an isomorphism, $E_\alpha\cap
E_\Delta$ is the unique point $Q\in E_\alpha$ such that
$\w{\varphi}(Q)=\infty$, hence $\Delta$ is the unique connected
component of ${\cal G}(\pi)\setminus \{\alpha\}$ such that
$\w{C}_h\cap E_\Delta \neq \emptyset$ and moreover we deduce that
$\nu_\gamma(Q_\alpha)=\nu_\gamma(h)$ if $\gamma\notin \Delta$.
\end{proof}

A close result holds when we change $Q_\alpha$ by whatever
element
$f \in R$:
\begin{proposition}\label{lem3}
 Let
$h$ and  $f$ be elements in $R$ such that
$\nu_{\alpha}(h)=\nu_{\alpha}(f)$ and assume that there exists a
connected component $\Delta$ of ${\cal G}(\pi)\setminus\{\alpha\}$
such that $E_{\Delta}$ contains $\w{C}_h\cap E$ and $\w{C}_f\cap
E$.
Then $ \nu_{\gamma}(f) = \nu_\gamma(h)$ for each $\gamma\notin
\Delta $ and there exists $c \in K$, $c \neq 0$, such that $
\nu_{\alpha}(f- c h)>\nu_{\alpha}(f)\; $.
\end{proposition}

\begin{proof}
We can assume that $\pi$ is an embedded resolution of the curve
$C_{fh}$, since the additional blowing-ups we need for it do not
modify the connected subset $T=\mathcal G(\pi)\setminus\Delta$.

If $\nu_{\gamma}(f)>\nu_{\gamma}(h)$ for some $\gamma\in T$, we
deduce that $\w{C}_f$ intersects $E_A$,  where $A$ is the maximal
connected subset of ${\cal G}(\pi) \setminus \{\alpha\}$ such that
$\gamma\in A$ and $\nu_\beta(f)>\nu_\beta(h)$ for every $\beta\in
A$. In particular, as $A\subset T$, $\w{C}_f$ intersects $E_{T}$,
which contradicts the hypothesis. Thus,
$\nu_\gamma(f)=\nu_\gamma(h)$ for every $\gamma\in T$ and the
lifting $\w{\varphi}$ of the rational function $\varphi=f/h$ to the
space $X$ is defined at every point of $E_T$. Moreover, $\w{\varphi}
|_{E_{\beta}} : E_{\beta}\to \PP^1$ cannot be surjective for
$\beta\in T$, since $\w{C}_h\cap E_T= \emptyset$ and $\w{C}_f\cap
E_T= \emptyset$. Therefore, there exists $c\in \PP^1_K$, $c\neq 0,
\infty$, such that $\w{\varphi}|_{E_T}\equiv c$. Then the lifting of
$(f-c h)/h$ vanishes on $E_T$ and in particular $\nu_\alpha(f-c
h)>\nu_\alpha(h)$ (in fact $\nu_\gamma(f-ch) > \nu_\gamma(h)$ for
every $\gamma\notin \Delta$).
\end{proof}

\begin{remark}
Let $f,h\in R$ such that $\nu_\alpha(f)=\nu_\alpha(h)$ and assume
that there exists $c \in \PP^1$, $c\neq 0, \infty$,  such that
$\nu_\alpha(f-ch)>\nu_\alpha(f)$. Then, the lifting of the rational
function $\varphi=f/h$ is constant and equal to $c$ along
$E_\alpha$. As a consequence, the strict transforms of $C_f$ and
$C_h$ intersect the same points of $E_\alpha$ and the same connected
components of ${\cal G}(\pi) \setminus \{\alpha\}$ (otherwise the
corresponding point of intersection in $E_\alpha$ must be  a zero or
a pole of $\varphi$).

On the other hand, let $f \in R$ be such that $\w{C}_f\cap E = P\in
E_\alpha$ is a smooth point of $E$, set $r=(E_{\alpha}, \w{C}_f)$
and pick $Q_{\alpha}$ by $P\in E_\alpha$. Then
$\nu_\alpha(Q_{\alpha}^r)=\nu_\alpha(f)$ and after some additional
blowing-ups we could apply the above proposition, proving the
existence of $c\neq 0,\infty$ such that
$\nu_\alpha(f-cQ_\alpha^r)>\nu_\alpha(f)$.
\end{remark}

\medskip

Now we recall some known facts about  curve singularities and
divisorial valuations.

\medskip
Let $\nu$ be a divisorial valuation,  $f\in R$ a general element of
$\nu$ defining a curve $C_f$ and $v$ the discrete valuation of the
fraction field of $R/(f)$ given by its integral closure. Let $h$ be
an element of $R$ such that the strict transform $\w{C}_h$ of $C_h$
by the minimal resolution of $\nu$ does not intersect the strict
transform $\w{C}_f$ of $C_f$. Then Equality~(\ref{eq00}) implies
that $\nu(h)=v(h)=(f,h)$. If $\w{C}_f\cap \w{C}_h\neq \emptyset$ one
can use a generic element $f'$ for which $\w{C}_{f'}\cap
\w{C}_h=\emptyset $ and $\nu(h)= v'(h)$, where $v'$ is the valuation
corresponding to $f'$.

The dual graph of the minimal resolution of a valuation
$\nu$ looks like that of Figure \ref{fig0}, where $\alpha(\nu)$ is
the vertex corresponding to the divisor $E_{N+1}=E_{\alpha(\nu)}$
defining the valuation $\nu$, $st_i$ stands for the star vertex of
the dead end $\rho_i$ and $\Gamma_i$ denotes the path from
$st_{i-1}$ to $\rho_{i}$.
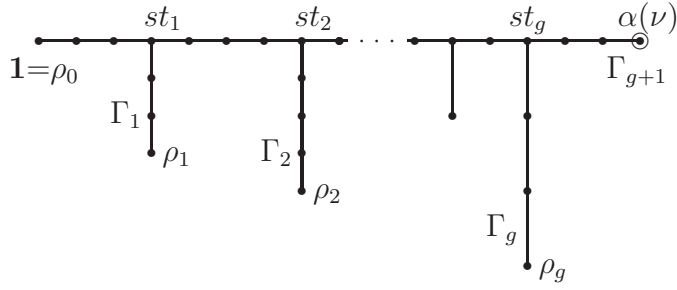
\begin{figure}[h]
$$
\unitlength=1.00mm
\begin{picture}(80.00,30.00)(-10,3)
\thicklines \put(-5,30){\line(1,0){41}}
\put(44,30){\line(1,0){31}} \put(38,30){\circle*{0.5}}
\put(40,30){\circle*{0.5}} \put(42,30){\circle*{0.5}}
\put(30,10){\line(0,1){20}} \put(50,20){\line(0,1){10}}
\put(60,0){\line(0,1){30}} \put(10,15){\line(0,1){15}} \thinlines
\put(20,30){\circle*{1}} \put(30,30){\circle*{1}}
\put(50,30){\circle*{1}} \put(60,30){\circle*{1}}
\put(65,30){\circle*{1}} \put(70,30){\circle*{1}}
\put(75,30){\circle*{1}} \put(75,30){\circle{2}}

\put(30,20){\circle*{1}} \put(60,20){\circle*{1}}
\put(60,10){\circle*{1}} \put(10,30){\circle*{1}}
\put(30,10){\circle*{1}} \put(50,20){\circle*{1}}
\put(60,0){\circle*{1}} \put(-5,30){\circle*{1}}
\put(0,30){\circle*{1}} \put(5,30){\circle*{1}}
\put(15,30){\circle*{1}} \put(25,30){\circle*{1}}
\put(35,30){\circle*{1}} \put(45,30){\circle*{1}}
\put(55,30){\circle*{1}} \put(10,25){\circle*{1}}
\put(10,20){\circle*{1}} \put(10,15){\circle*{1}}
\put(30,25){\circle*{1}} \put(30,15){\circle*{1}}
\put(35,30){\circle*{1}} \put(-9,25){{\bf 1}=$\rho_0$}
\put(11.5,14){$\rho_1$} \put(4.5,19){$\Gamma_1$}
\put(31.5,9){$\rho_2$} \put(24.5,14){$\Gamma_2$}
\put(61.5,-1){$\rho_g$} \put(54.5,4){$\Gamma_g$}
\put(9,32){$st_1$} \put(29,32){$st_2$} \put(57.5,32){$st_g$}
\put(72,32){$\alpha(\nu)$} \put(70.5,25){$\Gamma_{g+1}$}
\end{picture}
$$
\caption{The dual graph of a divisorial valuation.} \label{fig0}
\end{figure}

If $C_f$ is general for $\nu$ (i.e., $f$ is a general element of
$\nu$), then the dead ends of $\mathcal G(\pi)$, $\seq {\rho}0g$,
are also dead ends for the dual graph of $C_f$, which is the dual
graph of the minimal embedded resolution of $C_f$ together with an
arrow attached to the vertex, $\alpha(f)$, corresponding to the
component intersected by $\w{C}_f$. We will denote
$Q_i:=Q_{\rho_i}$ and we set $\betab_i = v(Q_i)$ ($0 \le i \le
g$), values which are usually called {\bf maximal contact values}
of the curve singularity $C_f$.
It is known that the set $\{\seq{\betab}0g\}$ and the Puiseux
pairs of $C_f$, and hence the equisingularity type of $C_f$,  are
equivalent data (e.g. $\betab_0$ is the multiplicity $m(f)$ of
$C_f$ at the origin). Moreover, $\{\seq{\betab}0g\}$ is a minimal
set of generators of the semigroup of values $S_{C_f} :=
\{v(h)\mid h\in R/(f)^*\}$ of $C_f$, $R/(f)^*$ denoting the
nonzero elements of the ring $R/(f)$.

For the divisorial valuation $\nu$ we have $\nu(Q_i) =\betab_i = v(Q_i)$ and so for the
semigroup of values of $\nu$, $S_{\nu}:= \{\nu(h)\mid h\in R\setminus\{0\}\}$, one has
$S_{\nu} = \langle\seq{\betab}0g\rangle = S_{C_f}$. Thus, arithmetical properties of $v$
are also true for the valuation $\nu$ (in \cite{tei}, the reader can see proofs  for the
main properties which we will use later in this context). For the sake of completeness we
will denote $\betab_{g+1} = \nu(Q_{\alpha(\nu)})$. It holds that $\betab_{g+1} =
e_{g-1}\betab_g +c$, where $e_{g-1}$ is the smallest positive integer such that $e_{g-1}
\betab_g\in \langle\seq{\betab}0{g-1}\rangle$ and $c\ge 0$ is the number of blowing-ups
needed to create $E_{\alpha(\nu)}$ after the divisor corresponding to $st_g$ was
obtained.
Thus, $\betab_{g+1}$ gives an additional datum to the semigroup of values $S_{\nu}$ which
permits to recover the dual graph of the divisorial valuation $\nu$ (see \cite{s-2}). The
element $\betab_{g+1}$ has an expression $\betab_{g+1}= \sum_{j=0}^{g}\lambda_j\betab_j$
with $\lambda_j\ge 0$ for $0 \le j \le g$, which is unique if we add some restrictions to
the coefficients $\lambda_j$. The case $c=0$ corresponds to $\alpha(\nu)=st_g$,  or
equivalently, to the case in which ${\cal G}(\pi)\setminus \{\alpha(\nu)\}$ has two
connected components, and in this case $\lambda_g=0$, thus, $\betab_{g+1} = e_{g-1}
\betab_g = \sum_{j=0}^{g-1}\lambda_{j}\betab_j$.

\medskip

For simplicity, we will often use the term ``monomial'' to  indicate a monomial in the set $\{Q_\rho\mid\rho \in \mathcal E\}$, that is, a finite product of the type $\prod_{\rho\in {\cal E}} Q_\rho^{\lambda_\rho}$ with $\lambda_\rho\in\mathbb Z_{\ge 0}$.

\begin{proposition}
\label{mainl} Let $\pi:X\to \spec R$ be a modification. Pick
$\alpha\in {\cal G}(\pi)$ and let $\Delta$ be a connected
component of ${\cal G}(\pi)\setminus \{\alpha\}$. Then, there
exists a monomial $q_{\Delta} = \pr_{\rho\in {\cal E}\cap \Delta}
Q_{\rho}^{\lambda_\rho} \;$ such that $\;\nu_{\gamma}(q_{\Delta})
= \nu_{\gamma}(Q_{\alpha}) \;$ if $\gamma\in {\cal G}(\pi)
\setminus \Delta$ and  $ \; \nu_{\gamma}(q_{\Delta}) >
\nu_{\gamma}(Q_{\alpha}) \; $ otherwise.
\end{proposition}

\begin{proof} We only need to
find a monomial $q_{\Delta} = \pr_{\rho\in {\cal E}\cap \Delta}
Q_{\rho}^{\lambda_\rho} \;$ such that $\;\nu_{\alpha}(q_{\Delta})
= \nu_{\alpha}(Q_{\alpha}) \;$, because it would satisfy  $\widetilde{C}_{q_{\Delta}}\cap E_{\Delta}\neq\emptyset$, and then,
by Proposition~\ref{lem1}, it  solves our problem.

Firstly, let us assume that $\pi: X\to \spec R$ is the minimal resolution of
$\nu_{\alpha}$. With the above notations,
 $\betab_{g+1}=\nu_{\alpha}(Q_{\alpha})= \sum_{i=0}^g \lambda_i
\betab_i$ and we have two possibilities depending whether
${\cal G}(\pi)\setminus \{\alpha\}$ is connected or not.
In the  first case, the
decomposition of $\betab_{g+1}$ provides the monomial $q_\Delta =
\pr_{i=0}^{g}Q^{\lambda_i}_{\rho_i}$. Otherwise ${\cal G}(\pi)\setminus \{\alpha\}$ has two
connected components; then, if  $[{\bf
1}]\in \Delta$, we have $\{\seq{\rho}0{g-1}\} = {\cal E} \cap
\Delta$ and the monomial is $q_\Delta = \pr_{i=0}^{g-1}Q^{\lambda_i}_{\rho_i}$
(recall that in this case $\lambda_g=0$),
and if  $[{\bf
1}]\notin \Delta$ we have $\{\rho_{g}\} = {\cal E}\cap \Delta$ and the
monomial is  $q_{\Delta}=
Q^{e_{g-1}}_{g}$.

In general, let us denote by $\pi': (Y,F)\to \spec R$ the minimal
resolution of $\nu_\alpha$ and let $\sigma: X\to Y$ be the
composition  of the sequence of point blowing-ups which produces $X$
starting from $Y$. We claim that if $\Omega$ is any connected component of ${\mathcal
G}(\pi)\setminus {\mathcal G}(\pi')$ such that $\sigma(E_{\Omega})=P\in
E_{\beta}$ is a smooth point of $F$,  then there exists a dead end $\rho\in \mathcal E\cap\Omega$
such that $\nu_{\gamma}(Q_{\rho})=\nu_{\gamma}(Q_{\beta})$ for any $\gamma\notin\Omega$. Indeed, it suffices to choose $\rho$ as an element of $\mathcal E\cap\Omega$ making minimal the number of blowing-ups needed to obtain it, since for this $\rho$,
the strict transform of
$Q_\rho$ by $\pi'$ is smooth and transversal to $F$ at $P$.

Now,  if $\sigma(E_{\Delta})$ is a smooth point $P\in E_{\alpha}$ of $F$,
the above construction applied to $\Delta$ gives $\rho\in {\cal E}\cap \Delta$ such that
$\nu_\alpha(Q_\rho)=\nu_\alpha(Q_\alpha)$, so we can choose
$q_\Delta=Q_\rho$.

Otherwise,
$\sigma(E_{\Delta})\subset  \k{F\setminus E_\alpha}$. In this case, if some dead end $\rho'$
of $\mathcal G(\pi')$ is not a dead end of $\mathcal G(\pi)$, then there exists a connected component
$\Omega$ of ${\cal G}(\pi)\setminus {\cal G}(\pi')$ such that
$\sigma(E_{\Omega}) = P\in E_{\rho'}$, $P$ a smooth point of $F$, and our claim gives a dead end $\rho$ of $\mathcal G(\pi)$  such that $\nu_\gamma(Q_{\rho}) = \nu_\gamma(Q_{\rho'})$ (and $\rho\in\Delta$ if $\rho'\in\Delta$). Hence, if $\{\seq{\rho'}0g\}$ are the dead ends of $\mathcal G(\pi')$, we can find $\{\seq{\rho}0g\}$ in $\mathcal E\cap\Delta$ such that $\nu_\gamma(Q_{\rho_i}) = \nu_\gamma(Q_{\rho'_i}) = \betab_i$ for $0\le i\le g$, and the the monomial is given as in the
case in which $\pi$ is the minimal resolution.
\end{proof}

\medskip

To end this section, assume that $h \in R$ is irreducible and
$\pi:
X\to \spec R$ a modification such that the strict transform of the
curve $C_h$ by $\pi$ only meets one irreducible component, that
we will denote $E_{\alpha(h)}$, of the exceptional divisor of $\pi$.

\begin{proposition}\label{mainp}
For any vertex $\beta \in {\cal G}(\pi)$,  there exists a monomial
$q= \pr_{\rho\in {\cal E}}Q_\rho^{\lambda_\rho}$ such that
$\nu_{\beta}(q)= \nu_\beta(h)$ and $\nu_{\gamma}(q)\ge
\nu_{\gamma}(h)$ for every $\gamma \neq \beta$. Moreover, if $\beta
\neq \alpha(h)$, then the vertices $\rho$ such that $\lambda_\rho
\neq
0$ belong to the connected component of $\alpha(h)$ in
${\cal G}(\pi)\setminus \{\beta\}$.
\end{proposition}

\begin{proof}
We can choose $Q_{\alpha(h)}$  through  $P=E_{\alpha(h)}\cap
\w{C}_h$. Setting $r=(E_{\alpha(h)}, \w{C}_h)$ we have
$\nu_\beta(Q_{\alpha(h)}^r)=\nu_\beta(h)$ for any $\beta\in {\cal
G}(\pi)$ (see the remark after Proposition \ref{lem3}). So, it
suffices to obtain $q$ for the case $Q_{\alpha(h)}$, since then
$q^r$ would solve the problem for $h$.

Now,  the monomial $q_\Delta$ given in Proposition \ref{mainl} for any connected
component $\Delta$ of ${\cal G}(\pi) \setminus\{\alpha(h)\}$ such that $\beta \not \in
\Delta$, if it exists, satisfies the requirements of the proposition. Moreover, if $
\beta \neq \alpha(h)$, then $\beta \not \in \Delta\cup \{\alpha(h)\}$ and this set is a
connected subset of ${\cal G}(\pi)\setminus \{\beta\}$, thus $\Delta\cup \{\alpha(h)\}$
is contained in a connected component of ${\cal G}(\pi)\setminus \{\beta\}$.

Otherwise, that is $\beta$ belongs to every  connected component of
${\cal G}(\pi)\setminus \{\alpha(h)\}$, $\alpha(h)$ must be a dead
end and we can take $q=Q_{\alpha(h)}$.
\end{proof}

\section{Semigroup of values}\label{value}

Let $V=\{\seq{\nu}1r\}$ be a finite set of $r \geq 1$ divisorial
valuations and denote by $\mathbb Z_{\ge 0}$ the set of nonnegative
integers. The {\bf semigroup of values} of $V$ is the additive
subsemigroup $S_{V}$ of $\mathbb Z^r_{\ge 0}$ given by $$ S_V =\{
\nuv(h):=\val{\nu}{r}{h} \; |\; h\in R\setminus\{0\} \}\; . $$
The {\bf minimal resolution} of $V$ is a modification $\pi: (X,E)\to
(\spec R, \ideal{m})$ such that, for each $i\in\{1,\ldots,r\}$,
$\nu_i$ is the $E_{\alpha(i)}$-valuation for an irreducible
component $E_{\alpha(i)}$ of the exceptional divisor $E$, and $\pi$
is minimal with this property. It is clear that a minimal resolution
of $V$ can  be recursively obtained by blowing-up $\spec R$ at
$\ideal{m}$ and any new obtained space  $X_i$ at the closed centers
of the valuations in $V$. The dual graph of $V$ is the dual graph
of $\pi$ with the vertices $\alpha(i)$
highlighted (for example, using a different draw for the point,
see Figure~\ref{fig0}).

Let $C=\bigcup\limits_{i=1}^r C_i$ be a reduced curve, with
components $C_1,\ldots, C_r$, defined by an element $f\in R$, and
denote by $R/(f)^*$  the set of nonzero divisors of the ring
$R/(f)$. The {\bf semigroup of values} $S_{C}$ of $C$ is the
additive subsemigroup of $\ZZ_{\ge 0}^r$ given by
$$
S_{C} : = \{\vv(g)=\val{v}rg \; |\; g \in R/(f)^* \}\;,
$$
where each $v_i$ is the valuation corresponding to $C_i$. Sometimes we will consider ``the value'' $\vv(h)$ (not in $S_C$)  of zero divisors of $R/(f)$, understanding $\nu_i(h)=\infty$ for $h$ in the ideal of $R$ defining $C_i$, and $n< \infty$ for any  $n\in \ZZ_{\ge 0}$.

The dual graph of $C$ is the dual graph of its minimal embedded
resolution, attaching an arrow, for each irreducible component
$C_i$ of $C$, to the exceptional component which meets the strict
transform on $X$ of $C_i$. The equisingularity type of $C$ (i.e.,
the set of Puiseux pairs for each branch $C_i$ of $C$ together
with the intersection multiplicities between pairs of branches)
and its dual graph, labelling each  vertex $\alpha$ with the
minimal number of blowing-ups needed to create $E_{\alpha}$,
$w(\alpha)$, are equivalent data.

\medskip

Let ${\mathcal G}$ and $S_V$ be the dual graph and the semigroup
of values of a set $V=\{\seq{\nu}1r\}$ of divisorial valuations,
$r>1$. A {\bf general curve} of $V$ is a reduced plane curve with
$r$ branches defined by $r$ different equations given by general
elements of each valuation $\nu_i$. An element $\mm\in S_V$ is
said to be {\bf indecomposable} if we cannot write $\mm=\nn+\kk$
with $\nn,\kk\in S_V\setminus\{0\}$.

For $1\le i\le r$ set $\alpha(i)=\alpha(\nu_i)$, and for each vertex $\rho\in {\cal E}$
denote by $\beta_\rho$ the nearest vertex to $\rho$ in $\Omega =
\bigcup\limits_{i=1}^{r}[{\bf 1},\alpha(i)]$ (i.e. $\beta_\rho = \max ( \Omega\cap [{\bf
1}, \rho]) $). Consider the set
$$
{\cal H} =
\{\bf 1\}\cup{\cal E}
\cup \left(\Omega
\setminus \left \{
\Gamma \cup \{\beta_{\rho} \; |\; \rho\in {\cal E}\} \right
\}\right)\;,
$$
where
$\Gamma=\bigcap\limits_{i=1}^r
[{\bf 1},\alpha(i)]$. Then we can state the following

\begin{theorem}
\label{semi} The set of indecomposable elements of the semigroup
of values $S_V$ is the set $\{\nuv(Q_\alpha)\; |\; \alpha\in
{\cal H}\}$. In particular, $S_V$ is finitely generated.
\end{theorem}

This theorem is the divisorial version of the next one which holds
for a reduced plane curve $C$ with $r$ branches (\cite{lon}).  In
it, we consider the dual graph of $C=\bigcup\limits_{i=1}^r C_i$ and
define ${\cal H}$ as above, and $\alpha(1),\ldots \alpha(r)$ are the
vertices with arrows, corresponding to the branches $C_1, \ldots,
C_r$ of $C$.

\begin{theorem}
\label{fel} The set of  indecomposable elements of the semigroup
$S_C$ is
$$
\{ \vv(Q_\alpha) \; |\; \alpha \in {\cal H}\} \cup \{
\vv(Q_{\alpha(i)}) + (0, \ldots,0, k,0, \ldots, 0) \; |\; i=1,
\ldots, r \quad k\geq 1 \},
$$
where  $k$ is in the $i$th component. $\;\; \Box$
\end{theorem}

\begin{proof}
Let us prove Theorem \ref{semi}. If $C=\bigcup\limits_{i=1}^r C_i$ is any general curve
of $V$, that is, $C_i$ is general for $\nu_i$,  then $S_V\subseteq S_{C}$, therefore, by
Theorem \ref{fel}, elements in the set $\{\nuv(Q_\alpha)\tq \alpha\in {\cal H}\}$ are
indecomposable. Conversely, given $h \in R$ such that $\nuv(h)$ is indecomposable in
$S_V$, choose a general curve $C$ of $V$ such that the strict transforms of $C$ and $C_h$
by the minimal resolution of $V$ do not intersect. So, from equality~(\ref{eq00}),
$\nuv(h)=\vv(h)$ and $\nuv(Q_{\alpha})=\vv(Q_{\alpha})$ for any vertex $\alpha$,
$\vv$ given by the valuations associated to $C$. Moreover, $h$
must be irreducible and by the proof of Theorem~\ref{fel}, \cite{lon}, $\vv(h)$
decomposes in $S_C$ as a sum of elements $\vv(Q_\gamma)$ with $\gamma \in \mathcal H$,
which proves that $\nuv(h)= \nuv(Q_\alpha)$ for some $\alpha\in {\cal H}$.
\end{proof}

\medskip

\begin{remark}
A consequence of Theorem \ref{semi} is  that the semigroup $S_V$
does not have conductor whenever $r>1$, that is, there is no element
$\delta\in S_V$ such that $\delta + \ZZ_{\ge 0}^r\subseteq S_V$.
However,
 the semigroup of values of a curve with
$r$ branches does have a conductor $\delta$ \cite[ Th. 2.7]{man},
and thus, it  cannot be finitely generated if $r>1$. In particular,
if $C$ is any general curve of $V$, $S_V \neq S_C$ when $r>1$
(recall that $S_V=S_C$ if $r=1$).
\end{remark}

\medskip

Considering the  ordering over $\ZZ^r$ given by $\nn \le
\mm\Leftrightarrow \mm-\nn\in\ZZ^r_{\ge 0}$, a finite set of
divisorial valuations $V=\{\seq{\nu}1r\}$ induces a multi-index
filtration of the ring $R$ by means of the valuation ideals
$J(\mm)$,
 $\mm=(\seq m1r)\in \ZZ_{\ge 0}^r$:
$$ J(\mm) : = \{g\in R \; |\; \nuv(g)\geq
\mm\}\;.
$$

For $J\subset \{1,\ldots,r\}$ denote by  $\ee_J$ the element of
$\ZZ_{\ge 0}^r$ whose $i$th component is equal to $1$ (respectively,
to $0$) if $i\in J$ (respectively, $i\not\in J$); denote
$\ee=\ee_{\{1,\ldots,r\}}$. We will use $\ee_i$ instead of
$\ee_{\{i\}}$.

We will denote $D(\mm)=J(\mm)/J(\mm+\ee)$  and $D_i(\mm)=J(\mm)/J(\mm+\ee_i)$ for $1\le
i\le r$. It is clear that the natural homomorphism $D(\mm)\rightarrow
D_1(\mm)\times\ldots\times D_r(\mm)$ is injective. For $h\in J(\mm)\setminus
J(\mm+\ee_i)$ we will denote
$\init_{\nu_i}(h)=h+J(\mm+\ee_i)\in D_i(\mm)$, and call it
{\bf the initial form of $h$} with respect to $\nu_i$.

When $r=1$, Nakayama's Lemma proves that for any $m\in \ZZ$, $D(m)$
is a finite dimensional $K$-vector space and, therefore, so are
$D(\mm)$ and $D_i(\mm)$ for $\mm\in \ZZ_{\geq 0}^r$. Set $d(\mm) =
\dim D(\mm)$ and  $d_i(\mm) := \dim D_i(\mm)$.

\medskip

In the sequel, we will set $B^i=\nuv(Q_{\alpha(i)})$,
$i=1,\ldots,r$. Let $f\in R$ be such that
$\nu_i(f)=\nu_i(Q_{\alpha(i)})$ (remember that $\alpha(i)$ denotes
the vertex $\alpha(\nu_i)$ corresponding to the divisor that
defines $\nu_i$). Then by Proposition~\ref{lem1}, $\nu_j(f)\ge
\nu_j(Q_{\alpha(i)})$ for $j=1,\ldots,r$. Moreover, by
Lemma~\ref{lem12}, if $\w{C}_f\cap
\w{C}_{Q(\alpha(i))}=\emptyset$,
there exists a unique point $P(f)$ in $E_{\alpha(i)}$ mapped to
$\infty$ by the lifting of the rational function
$\varphi=Q_{\alpha(i)}/f$, namely, $P(f)=\w{C}_f\cap
E_{\alpha(i)}$ if $\w{C}_f\cap E_{\alpha(i)}\neq \emptyset$ and
$P(f)=E_{\Delta}\cap E_{\alpha(i)}$ if $\w{C}_f\cap
E_{\alpha(i)}=\emptyset$ and $\Delta$ is the connected component
of $G(\pi)\setminus\{\alpha(i)\}$ such that $\w{C}_f\cap
E_{\Delta}\neq \emptyset$. Furthermore, we denote
$P(f)=\w{C}_{Q_{\alpha(i)}} \cap E_{\alpha(i)}$ whenever $\w{C}_f
\cap \w{C}_{Q_{\alpha(i)}}\neq\emptyset$.

\begin{proposition}\label{isom}
The map $\Phi: \PP D_i(B^i)\to E_{\alpha(i)}$ from the
projectivization of the vector space $D_i(B^i)$ to the exceptional
component $E_{\alpha(i)}$, which sends the class
$\init_{\nu_i}(f)$ to  $P(f)$, is an isomorphism. In particular,
$d_i(B^i)=2$ and a basis of $D_i(B^i)$ is given by the initial
forms of two elements $f$ and $g$ such that $P(f)\neq P(g)$ (e.g.,
two $Q_{\alpha(i)}$ elements at two different points in
$E_{\alpha(i)}$).

\end{proposition}

\begin{proof}
First of all, we assert that $\Phi$ is well-defined. In fact,
given $f,g\in J(B^i)\setminus J(B^i+\ee_i)$ such that
$\init_{\nu_i}(f)=\lambda\init_{\nu_i}(g)$, that is,
$\nu_i(f-\lambda g)
> \nu_i(f)=\nu_i(g)=\nu_i(Q_{\alpha(i)})$
 for some $\lambda \in K\setminus\{0\}$, the liftings $\w{\varphi}_1$ and $\w{\varphi}_2$ of
the rational functions $f/Q_{\alpha(i)}$ and $g/Q_{\alpha(i)}$ are
defined in $E_{\alpha(i)}$, hence  the lifting of $(f-\lambda
g)/Q_{\alpha(i)}$ is also defined and it vanishes in
$E_{\alpha(i)}$. This
means that $\w{\varphi}_1=\lambda\w{\varphi}_2$ in $E_{\alpha(i)}$
and so $P(f)=P(g)$.

It is evident that $\Phi$ is surjective, let us see that it is
injective. Take $f,g\in J(B^i)\setminus J(B^i+\ee_i)$ such that
$P(f)=P(g)$. If $\w{C}_f\cap \w{C}_{Q_{\alpha(i)}}=\emptyset$,
then $\w{C}_g\cap \w{C}_{Q_{\alpha(i)}}=\emptyset$, and, perhaps
with some additional blowing-ups, we are in the situation of
Proposition \ref{lem3}, so there exists $\lambda\in
K\setminus\{0\}$ such that $\nu_i(f-\lambda g)>\nu_i(f)$, that is,
$\init_{\nu_i}(f)=\lambda\init_{\nu_i}(g)$ as we want.
Otherwise, $\w{C}_g\cap \w{C}_{Q_{\alpha(i)}}\neq\emptyset$ and
since $\nu_i(f)=\nu_i(g)=\nu_i(Q_{\alpha(i)})$, $f, g$ and
$Q_{\alpha(i)}$ are irreducible, smooth and transversal to $E_{\alpha(i)}$. Making an additional blowing-up
at the point $P=\w{C}_g\cap E_{\alpha(i)}=\w{C}_f\cap E_{\alpha(i)}$, we
can conclude, applying  again Proposition \ref{lem3}, that $\init_{\nu_i}(f)=\lambda\init_{\nu_i}(g)$
for some $\lambda\in K\setminus\{0\}$.
\end{proof}

\begin{proposition}\label{cor10}
The map $\w{\Phi}: \PP D(B^i)\to E_{\alpha(i)}$ which sends the
class of $f$ to $P(f)$, is an isomorphism. In particular $d(B^i)=2$.
\end{proposition}

\begin{proof}
The result is a consequence of the Proposition~\ref{isom} and of the next lemma.
\end{proof}

\begin{lemma}\label{lem13}
The natural homomorphism $D(B^i)\to D_i(B^i)$ is an isomorphism.
\end{lemma}

\begin{proof}
Let $f\in R$ be such that $\nu_j(f)\ge B^i_j=\nu_j(Q_{\alpha(i)})$
for every $j\in\{1,\ldots, r\}$ and $\nu_i(f)>
B^i_i=\nu_i(Q_{\alpha(i)})$. We need to prove  that $\nu_j(f) >
B^i_j$ for any $j$.

Denote by $\Delta$ the maximal connected subset of $\mathcal
G(\pi)$ such that $\alpha(i)\in \Delta$ and
$\nu_{\beta}(f)>\nu_{\beta}(Q_{\alpha(i)})$ for every $\beta\in
\Delta$. Notice that the lifting $\w{\varphi}$ of the function
$\varphi=f/Q_{\alpha(i)}$ is defined and it is identically $0$ in
$E_{\Delta}$, in particular $E_{\beta}$ is not dicritical for the
pencil $L=\{\lambda f+ \mu Q_{\alpha(i)}\mid \lambda,\mu\in K\}$
for any $\beta\in \Delta$. Let us see that $\Delta=\mathcal
G(\pi)$, which proves the lemma.

Otherwise, we could choose a divisor $E_{\beta}$ such that
$E_{\beta}\cap E_{\Delta}\neq\emptyset$ and $\beta\notin \Delta$,
that is, $\nu_{\beta}(f)\le \nu_{\beta}(Q_{\alpha(i)})$. By making
some additional blowing-ups, we can suppose that in fact
$\nu_{\beta}(f)= \nu_{\beta}(Q_{\alpha(i)})$, then $\w{\varphi}$
is defined and it is not constant in $E_{\beta}$, so it is
dicritical for $L$. Hence, there exists a point $P\in E_{\beta}$,
$P\neq E_{\beta}\cap E_{\Delta}$, such that
$\w{\varphi}(P)=\infty$, and this means that $Q_{\alpha(i)}$ meets
either $E_{\beta}$ at $P$ or $E_{\Delta'}$, $\Delta'$ being the
connected component  of $P$ in $\mathcal
G(\pi)\setminus\{\beta\}$. But both things are impossible, as
$Q_{\alpha(i)}$ only meets $E$ at $E_{\alpha(i)}$, and
$\alpha(i)\in \Delta\subset \mathcal G(\pi)\setminus \Delta'$.
\end{proof}

\medskip

The following two lemmas are devoted to prove Theorem
\ref{str_semi} which gives an  explicit description of the
semigroup $S_V$ and clarifies the special role of the elements
$B^1, \ldots, B^r$. Fix $\mm\in \ZZ^r$ and $i$ such that $1\le
i\le r$.

\begin{lemma}\label{lem5}
$d_i(\mm)\ge 2$ if and only if $d_i(\mm-B^i)\ge 1$. Moreover, if
$\mm \in S_V$, then $d_i(\mm)\ge 2$ if and only if $\mm-B^i\in
S_V$.
\end{lemma}

\begin{proof}
If $d_i(\mm-B^i)\ge 1$, take $h\in J(\mm-B^i)\setminus J(\mm-B^i+\ee_i)$ and choose a
basis $\{\init_{\nu_i}(h_1), \init_{\nu_i}(h_2)\}$ of $D_i(B^i)$. Then
$\init_{\nu_i}(h\,h_1), \init_{\nu_i}(h\,h_2)$ are  linearly independent vectors in
$D_i(\mm)$.

Conversely, pick $h_1, h_2\in J(\mm)\setminus J(\mm+\ee_i)$ whose
classes in $D_i(\mm)$ are linearly independent. Every nonzero
function of the pencil $L$ generated by $h_1$ and $h_2$, $L =
\{\lambda h_1 + \mu h_2 \tq \lambda, \mu\in K\}$, satisfies
$\nu_i(\lambda h_1 + \mu h_2) =m_i$, so $E_{\alpha(i)}$ is
dicritical for $L$. Therefore, the restriction to $E_{\alpha(i)}$ of
the lifting to $X$, $\w{\varphi}$, of the rational function
$\varphi=h_1/h_2$ defines a $s$ to $1$ surjective morphism from
$E_{\alpha(i)}$ onto $\PP^1_K$. Then, a generic fiber $h = \lambda
h_1+\mu h_2$ of $L$ can be factorized in $R$ as $h = h'
\prod_{l=1}^s g_l $, where the $g_l$ are irreducible, $g_l\neq g_j$
when $l \neq j$ and the strict transform of each curve $C_{g_l}$ is
smooth and transversal to $E_{\alpha(i)}$ in a smooth point.
Therefore $\nuv(g_l)=B^i$ and  $h/g_i\in J(\mm-B^i)$ but
$h/g_i\notin J(\mm-B^i+\ee_i)$.

Moreover, if $\mm\in S_V$ then $h_2$ can be chosen in such a way that $\nuv(h_2)=\mm$ and
so for $\lambda$ and $\mu$ generic we have $\nuv(h)=\mm$ and  $\nuv(h/g_i) = \mm-B^i\in
S_V$.
\end{proof}

\begin{lemma}\label{lem6}
\
\begin{enumerate}
\item
If $\mm\in S_V$ and $j\neq i$
then
$d_i(\mm+B^j) = d_i(\mm)$.
\item
If $d_i(\mm) \neq 0$ then $d_i(\mm+B^i) = 1 + d_i(\mm)$.
\end{enumerate}
\end{lemma}

\begin{proof}
First, we will prove that if $j \neq i$ and $\mm \in S_V$ then the
multiplication by $Q_{\alpha(j)}$ provides  a linear bijective map
$\psi:D_i(\mm)\to D_i(\mm+B^j)$. Clearly it is injective, let us see
that it is also surjective. Pick an element $f \in R$ such that
$\nuv(f)=\mm$ and take $h\in J(\mm+B^j)\setminus J(\mm+B^j+\ee_i)$.
Notice that $\nu_j(h)\geq \nu_j(fQ_{\alpha(j)})$ and so $\nu_j(h -
\lambda fQ_{\alpha(j)})\geq \nu_j(f Q_{\alpha(j)})$ for $\lambda\in
K$.

If $\nu_j(h - \lambda fQ_{\alpha(j)})>\nu_j(f Q_{\alpha(j)})$, for
some $\lambda \in K $, then there exists an irreducible component
$g$ of $h$ such that the strict transforms  of  $C_{g}$ and
$C_{Q_{\alpha(j)}}$ by the minimal resolution of $V$ intersect
$E_{\alpha(j)}$ at the same point, and then $\init_{\nu_j}(g)= b
\cdot\init_{\nu_j}(Q_{\alpha(j)})^c$ for some $c\ge 1$ and $b\in
K\setminus \{0\}$ (see the remark after Proposition \ref{lem3}).
Thus $h' = bhQ_{\alpha(j)}^{c}/g$ and $h$ have the same value and
initial form with respect to $\nu_k$ for $1 \leq k \leq r$. In
particular, $\init_{\nu_i}(h)=\init_{\nu_i}(h')\in \im \psi$.

Otherwise, $\nu_j(h - \lambda f Q_{\alpha(j)}) = \nu_j(f Q_{\alpha(j)})$ for all
$\lambda\in K$ and then $E_{\alpha(j)}$ is a dicritical divisor of the pencil generated
by $h$ and $f Q_{\alpha(j)}$. Thus, for a generic $\lambda$, $h - \lambda f
Q_{\alpha(j)}$ has an irreducible component $g$ such that $\w{C}_{g}$ is smooth and
transversal to $E_{\alpha(j)}$ at a smooth point. As $i\neq j$, by Proposition
\ref{lem3},  there exists $b\in K\setminus\{0\}$ such that $\init_{\nu_i}(g) = b \,
\init_{\nu_i}(Q_{\alpha(j)})$. Then  $h'=(h - \lambda f Q_{\alpha(j)})/g\in
J(\mm)\setminus J(\mm+\ee_i)$ and $\init_{\nu_i}(g\, h') = \init_{\nu_i}(b\,
Q_{\alpha(j)} h') \in \im \psi$. Hence $\init_{\nu_i}(h) = \lambda\init_{\nu_i}(f
Q_{\alpha(j)})+ \init_{\nu_i}(g\, h') \in\im \psi$.

\medskip

Now, we will prove 2. Assume $j=i$ and pick elements $\seq h1s \in
J(\mm)\setminus J(\mm+\ee_i)$  such that the set
$\{\init_{\nu_i}(h_l) | 1\le l\le s\}$ is a basis of $D_i(\mm)$.
Take  an irreducible element $g\in R$  such that  $\w{C}_{g}$ is
smooth and transversal to $E_{\alpha(i)}$ at a smooth point $P$,
$\w{C}_{g}\cap\w{C}_{Q_{\alpha(i)}}=\emptyset$ and
$\w{C}_{g}\cap\w{C}_{h_s}=\emptyset$. Then
$\init_{\nu_i}h_1g,\ldots,
\init_{\nu_i}h_sg,\init_{\nu_i}h_sQ_{\alpha(i)}$ are linearly
independent in the vector space $D_i(\mm+B^i)$, because in other case we could find $h=\sum\lambda_ih_i\in
J(\mm)\setminus J(\mm+\ee_i)$ and $\lambda\neq 0$ with $\nu_i(hg-\lambda h_sQ_{\alpha(i)})>\nu_i(hg)$ and then $\w{C}_{h_sQ_{\alpha(i)}}$ must intersect $E_{\alpha(i)}$ at the point $P$ (see again the remark
after Proposition \ref{lem3}), in contradiction with the election of $g$. Hence,
$d_i(\mm+B^i)\geq d_i(\mm)+1$.

To finish the proof, it suffices to show that if $d_i(\mm+B^i) =
t\ge 2$ then $d_i(\mm)\ge t-1$.  In fact, let $\{\init_{\nu_i}g_1,
\ldots, \init_{\nu_i}g_t\}$ be a basis of $D_i(\mm+B^i)$ and
consider the family of pencils $L_k = \{\lambda g_1 + \mu g_k\}$,
$2 \leq k \leq t$. Fix a smooth point $P\in E_{\alpha(i)}$ in such
a way that $P$ is non-critical for all the pencils $L_k$. For each
$k=2,\ldots, t$, let $\lambda g_1+\mu g_k = \varphi_k g_k'$ be the
fiber of $L_k$ corresponding to $P$ and $\varphi_k$ the unique
irreducible component of such fiber by $P$. In this way, all the
initial forms of $\varphi_k$ are equal (up to product by
constants). Set ${\cal B}=\{g_1, \varphi_2 g'_2, \ldots, \varphi_t
g'_t\}$. Then, for generic $P$, $\init_{\nu_i}(\cal{B})$ is  a
basis of $D_i(\mm+B^i)$, and $\init_{\nu_i}(g'_2), \ldots,
\init_{\nu_i}(g'_t)\in D_i(\mm)$ are linearly independent
elements. Thus $d_i(\mm)\ge t-1$ and the proof is finished.
\end{proof}

\begin{theorem}\label{str_semi}
For any  $\mm\in S_V$ there exist unique $\seq a1r\in \ZZ_{\ge 0}$
and $\nn\in S_V$ such that
\begin{enumerate}
\item
$\mm = \nn + a_1 B^1 + \cdots + a_r B^r$.
\item
$d_i(\nn)=1$ for every $i=1,\ldots,r$.
\end{enumerate}
In fact $a_i = \max \{k \in \ZZ_{\geq 0}
\mid  \mm - k B^i \in S_V\} = d_i(\mm)-1$  for $i=1,\ldots,r$.
\end{theorem}

\begin{proof}
Assume the existence of the values $a_1,\ldots, a_r, \nn$, then, by
Lemma \ref{lem6}, $1=d_i(\nn)=d_i(\mm-a_iB^i)=d_i(\mm)-a_i$, and by
Lemma \ref{lem5}, $\nn-B^i\notin S_V$, so $a_i= \max\{k\in \ZZ_{\geq
0} \mid \mm - kB^i \in S_V\}$, and we have the uniqueness. We also
have $a_i=d_i(\mm)-1$.

For the existence, define $a_i= \max\{k\in \ZZ_{\geq 0} \mid \mm - k
B^i \in S_V\}$ and $\nn = \mm-\sum_ka_kB^k$. To prove $\nn\in S_V$
it suffices to prove that if $\mm-B^i\in S_V$ and $\mm-B^j\in S_V$
then $\mm-B^i-B^j\in S_V$. The conditions $\mm-B^i\in S_V$ and
$\mm-B^j\in S_V$ imply, by Lemmas~\ref{lem5} and \ref{lem6} that
$d_j(\mm-B^i) = d_j(\mm) \ge 2$ and hence that $\mm-B^i-B^j\in S_V$.
\end{proof}

\begin{corollary}
Given a modification $\pi$ and the family of all the valuations associated to the components $\{E_1,\ldots, E_s\}$ of the exceptional divisor of $\pi$, $W=\{\nu_1,\ldots, \nu_s\}$, it holds that
$S_W=\langle B^1,\ldots, B^s \rangle\cong\mathbb Z^s_{\geq 0}$. $\; \; \Box$
\end{corollary}

\begin{remark}
The above corollary, established here as a consequence of
Theorem~\ref{str_semi}, was already
known, since the determinant of the intersection matrix of
the components $\{\seq E1s\}$ of the exceptional divisor of
$\pi$, $M =  (E_i\cdot E_j)$, is  $-1$,
the $s$  rows of $A = -M^{-1}$ are exactly the values
$\{B^1,\ldots, B^s\}$ and  $S_{W}=\{\mm\in \ZZ_{\geq
0}^s\mid-\mm M\geq 0\}$. Thus the
semigroup is the free semigroup generated by the vectors
$B^1, \ldots, B^s$.

In the general case, valuations in $V$ are those corresponding
to a subset $L$ of $\{1,\ldots, s\}$, $V=V_L=\{\nu_l\mid l\in
L\}$, and then $S_{V_L}$ is the projection over $\mathbb Z_{\ge
0}^{|L|}$ (that is, over the coordinates in $L$) of  the semigroup
$S_{W}$, so it is  contained
in the convex polyhedral cone in $\mathbb
R^{|L|}_{\geq 0}$ generated by the elements $\{B_l\mid l\in L\}$.
\end{remark}

\section{Graded algebra and generating sequences}\label{gradedalg}

Throughout this section, we will consider a nonempty finite set of
divisorial valuations $V=\{\seq{\nu}1r\}$ and we will use the
notations of the above sections. The  {\bf graded $K$-algebra}
associated to $V$ is defined to be
\[
\gr_{V}R : = \bigoplus_{\mm \in  \ZZ_{\ge 0}^r}
\frac{J(\mm)}{J(\mm + \ee)}\; .
\]

Set $\Lambda = \{u_j\}_{j\in J}$  a subset of the maximal ideal
\ideal{m} of $R$. A monomial in $\Lambda$ is a product $\pr_{j\in
J}u_j^{\gamma_j}$ with $\gamma_j\in \ZZ_{\ge 0}$ and $\gamma_j=0$
except for a finite subset of $J$. Let ${\cal M}(\Lambda)$ denote
the set of monomials in $\Lambda$, we will say that $\Lambda$ is
a
{\bf generating sequence} of $V$ if for each $\mm\in \ZZ_{\ge 0}^r$ the ideal
$J(\mm)$ is generated by ${\cal M}_{\mm}(\Lambda):= {\cal
M}(\Lambda)\cap J(\mm)$. In particular, $\Lambda$ is a system of
generators of $\ideal m$.

A generating sequence $\Lambda$ of $V$ is said to be minimal whenever each proper subset
of $\Lambda$ fails to be a generating sequence. In this case $V$ is said to be
\textbf{monomial with respect to $\Lambda$}. Generating sequences of a family $V$ and its
graded algebra $\gr_{V}R $ are closely related, as the following result (proved in
\cite{c-g} in a more general context)  shows:

\begin{theorem}
\label{belga} Assume that there exists a finite generating sequence
for some valuation of $V$. Then, a  system of generators $ \Lambda =
\{ u_{j} \}_{j \in J}$ of the maximal ideal $\ideal{m}$ is a
generating sequence of $V$ if and only if the $K$-algebra $\gr_{V}R$
is generated by  the set $\union_{j\in J} [u_{j}]$, where $[u_{j}]$
denotes the cosets that $u_{j}$ defines in $\gr_{V}R$. $\; \; \Box$
\end{theorem}

It is convenient to clarify the sense of the notation $[u]$ in the above theorem: if
$u\in \ideal{m}$ and $\mm=\nuv(u)$, then $u\in J(\nn)$ for any $\nn \le \mm$. Denote
$[u]_{\nn}:= u + J(\nn+\ee)$. So, $[u]_{\nn}\neq 0$ if, and only if, $\nn+\ee \nleq \mm$
(that is, $n_i=m_i$ for some index $i\in \{1,\ldots,r\}$). Then, $[u]$ in Theorem
\ref{belga} means $ [u]:=\{[u]_\nn \; | \; \nn\le \mm \text{ and } \nn+\ee \nleq \mm
\}\;$.

\medskip

Denote by ${\cal E}$ the set of dead ends of the dual graph of $V$ and fix an element
$Q_\rho\in R$ for each $\rho\in {\cal E}$ . Set
$$
\Lambda_{\cal E} = \{Q_{\rho}\; | \; \rho\in {\cal
E}\}\; .
$$
Next result is the analogous of Proposition~\ref{mainp} for
initial forms of elements
in $R$.

\begin{proposition}\label{truemain}
Given $h\in R$ and $i\in \{1,\ldots,r\}$, there exists a linear combination  of monomials
$q = \sum a_\lambda q^{\lambda}$, $q^\lambda = \prod_{\rho\in {\cal E}}
Q_\rho^{\lambda_\rho}$, such that $\nu_i(q)=\nu_i(h)$, $\nu_i(h-q)>\nu_i(h)$ and
$\nu_j(q)\ge \nu_j(h)$ for every index $j \neq i$.
\end{proposition}

\begin{proof}
Note that the condition $\nu_i(h-q)>\nu_i(h)$ is equivalent to
$\init_{\nu_i}(q)=\init_{\nu_i}(h)$. Thus, it suffices to prove the result for $h$
irreducible. Let $\pi$ be the minimal modification such that $\pi$ is a resolution of $V$
and the strict transform $\w{C}_h$ of $C_h$ by $\pi$ only meets one irreducible component
of the exceptional divisor of $\pi$, $E_{\alpha(h)}$.

If $\alpha(h) \neq \alpha(i)$ then, by Proposition \ref{lem3},
there exists $\lambda$ such  that $\lambda q$, $q$ being the
monomial constructed in Proposition \ref{mainp}, satisfies the
result. Assume that $\alpha(h)=\alpha(i)$ and choose
$Q_{\alpha(h)}$ such that $\w{C}_{Q_{\alpha(h)}}$ goes through the
intersection point $ E_\alpha \cap \w{C}_h$. Denoting
$m=(E_{\alpha(h)}, \w{C}_h)$, we have $\nu_j(h) =
\nu_j(Q_\alpha^m)$ for $1 \leq j \leq r$,  and $\init_{\nu_i} (h)
= \lambda \, \init_{\nu_i}(Q_\alpha^m)$ for some $\lambda \in
K\setminus\{0\}$ (see the remark after Proposition \ref{lem3}), so
we only need to prove the statement for $h = Q_{\alpha(i)}$.

Let $\seq{\Delta}0s$ be the connected components of ${\cal
G}(\pi)\setminus\{\alpha(i)\}$ and  $q_{\Delta_i}$, $0\le i\le s$
the monomial constructed in Proposition~\ref{mainl} for
$\Delta_i$. If $s\ge 1$, by Proposition~\ref{isom}, the classes of
any pair $q'$ and $q''$ of such monomials are a basis of
$D_i(\nuv(Q_{\alpha(i)}))$, thus $\init_{\nu_i}(h)=
\lambda\init_{\nu_i}(q')+\mu\init_{\nu_i}(q'')$ for some
$\lambda,\mu\in K$ and the linear combination $q= \lambda q' +\mu
q''$ satisfies the requirements of the statement. Finally, if
$s=0$, the vertex $\alpha(i)$ is an end vertex and we can use
$Q_{\alpha(i)}$ together with $q_{\Delta_0}$ to have a basis of
$D_i(\nuv(Q_{\alpha(i)}))$.
\end{proof}

\bigskip
Now, let $C$  be a reduced plane curve with $r$ branches,
$C_1,\ldots, C_r$, and local ring ${\cal O}=R/(f)$, and denote $\vv
: =(\seq v1r)$, where $v_i$ is the valuation associated to $C_i$. We
will say that $\Lambda\subset \ideal{m}$ is a generating sequence of
$C$ if the valuation ideals $J^C(\mm)= \{g \in {\cal O} | \vv(g)
\geq \mm\}$ are generated by the images in ${\cal O}$ of the
monomials in $\Lambda$. We will set $c(\mm):=\dim C(\mm)$, where
$C(\mm)=\dfrac{J^C(\mm)}{J^C(\mm + \ee)}$ is the corresponding
vector space of initial forms. Finally, we define the graded
$K$-algebra of ${\cal O}$ as
$$
{\rm gr}{\cal O} : = \bigoplus_{\mm \in  \ZZ_{\ge 0}^r}
\frac{J^C(\mm)}{J^C(\mm + \ee)}\; .
$$

\medskip
Denote by ${\cal E}$ the set of dead ends of the dual graph of $C$
and let $f_i$ be an element in $R$ that gives an equation for
$C_i$, ($1 \le i \le r$). We define
$$
\Lambda_{\k{\cal E}} = \{Q_{\alpha}  \; |\; \alpha\in {\cal E}\}
\cup \{\seq f1r\}\; .
$$
where we do not include $f=f_1$ if $r=1$.

Reduced curves can be approached by finite sets of divisorial
valuations. Indeed, denote by $\pi^{(0)}: X^{(0)}\to \spec R$ the
minimal embedded resolution of the curve $C$,  and by
$\pi^{(k)}:(X^{(k)},E^{(k)})\to \spec R$ the composition of
$\pi^{(k-1)}$ with $r$ additional blowing-ups, one at each point
where the strict transform of $C$ intersects $E^{(k-1)}$. Set
$\nu^{(k)}_i$ the $E_{\alpha^{(k)}(i)}$-valuation,
$E_{\alpha^{(k)}(i)}$ being the irreducible component of the
exceptional divisor $E^{(k)}$ intersected by the strict transform
of the branch $C_{i}$. Then, the sequence
$V^{(k)}=\{\seq{\nu^{(k)}}1r\}$ approaches $C$, in the sense that
for any element $h\in R$ which is not divisible by any $f_i$,
$\nu^{(k)}_i(h)=(f_i,h)=v_i(h)$ for $k\gg 0$ and $1\le i\le r$.

For $1\le i\le r$, denote by $\alpha^{(0)}(i)$ the vertex of $\mathcal G(\pi^{(0)})$ such
that the strict transform of the branch $C_i$ meets $E_{\alpha^{(0)}(i)}$. Then, for any
$k>0$, the graph of $V^{(k)}$, $\mathcal G(\pi^{(k)})$, is obtained by adding $r$
vertices $\alpha^{(k)}(1), \ldots, \alpha^{(k)}(r)$ (corresponding to the components
$E_{\alpha^{(k)}(i)}$) to $\mathcal G(\pi^{(k-1)})$, each $\alpha^{(k)}(i)$ adjacent to
$\alpha^{(k-1)}(i)$. Denoting by $\mathcal E^{(k)}$  the set of dead ends of $\mathcal
G(\pi^{(k)})$, it is clear that  $\mathcal E^{(k)}=(\mathcal E^{(k-1)}\setminus \{\alpha^{(k-1)}(1), \ldots,
\alpha^{(k-1)}(r)\})\cup \{\alpha^{(k)}(1), \ldots, \alpha^{(k)}(r)\}$ for $k\ge 1$.

Moreover, for each $k\ge 0$, the strict transform by $\pi^{(k)}$ of
the branch $C_i$ is smooth and meets $E_{\alpha^{(k)}(i)}$
transversally at a nonsingular point, so we can choose
$Q_{\alpha^{(k)}(i)}=f_i$. In this way, for every $k\ge 1$, when $r>1$ we have
$\Lambda_{\mathcal E^{(k)}}= \Lambda_{\k{\cal E}}$ ,  and   $\Lambda_{\mathcal E^{(k)}}= \Lambda_{\k{\cal E}}\cup \{f_1\}$ in the case $r=1$.

Note that $\mathcal G(\pi^{(k)})\subset \mathcal G(\pi^{(k+1)})$,
and the (infinite) graph obtained by blowing-up every infinitely
near point of $C$, is exactly the union $\union_{k\ge 0}\mathcal
G(\pi^{(k)})$. Analogously, if $S_{V^{(k)}}$ denotes the value
semigroup of the set $V^{(k)}$, one  gets the inclusion chain
$S_{V^{(0)}}\subseteq S_{V^{(1)}}\subseteq  \cdots$ and  the
equality $S_C = \union_{k\ge 0} S_{V^{(k)}}$.

\begin{theorem}\label{maint}
Let $V$ and $C$ be as above. Then, $\Lambda_{\cal E}$
($\Lambda_{\k{\cal E}}$, respectively) is a  minimal generating
sequence of $V$ ($C$, respectively).
\end{theorem}

\begin{proof}
Let us prove first the result for the curve $C$. Consider the
sequence $V^{(k)}$ as explained above, in such a way that
$\Lambda_{\mathcal E^{(k)}}= \Lambda_{\k{\cal E}}$ for every $k$ if $r>1$
and $\Lambda_{\k{\cal E}}=\Lambda_{\mathcal E^{(k)}}\setminus \{f_1\}$ if $r=1$.
Let $h\in R$ be such that $h\notin (f)$. We claim that there exists
a monomial $q_1$
in $\Lambda_{\k{\mathcal E}}$
such
that
$v_1(h)=v_1(q_1)$ and
$v_i(h)\le v_i(q_1)$ for $i=1,\ldots,r$. To prove the claim it is enough to find such a monomial for each irreducible component of $h$, so assume $h$ irreducible. Moreover, if $h=f_i$ for some $i$, then we can take $q_1=f_i$. Otherwise,  take $k>>0$
such that the strict transform of $C_h$ by $\pi^{(k)}$ does not
intersect any of the components $E_{\alpha^{(k)}(i)}$, $1\le i\le
r$.
Then, $\nu^{(k)}_i(h)=v_i(h)$ for $i=1,\ldots,r$, and applying
Proposition~\ref{mainp} to the set $V^{(k)}$ we find a monomial
$q_1=\pr_{\rho \in \mathcal E^{(k)}}Q_\rho^{\lambda_\rho}$ such
that
$\nu_1^{(k)}(h)=\nu_1^{(k)}(q_1)$,  $\nu_i^{(k)}(h)\le
\nu^{(k)}_i(q_1)$ for $i=1,\ldots,r$ and
$\lambda_{\alpha^{(k)}(1)}=0$. In particular, $f_1=Q_{\alpha^{(k)}(1)}$ does not appear in the expression of $q_1$, so $q_1$ is a monomial in $\Lambda_{\k{\mathcal E}}$ even in case $r=1$. Moreover, $\nu^{(k)}_1(q_1)=v_1(q_1)$ and
$\nu^{(k)}_i(q_1)\le v_i(q_1)$ for any $i$, therefore $v_1(h)=v_1(q_1)$ and
$v_i(h)\le v_i(q_1)$ for $i=1,\ldots,r$.

So, for $h\in R\setminus (f)$ we have the monomial $q_1$ of the claim, and there exists a nonzero constant $a_1$ with $v_1(h-a_1q_1)> v_1(h)$ and
$v_i(h-a_1q_1)\ge v_i(h)$ for $i \geq 2$. The same claim can be applied  to an
index (if it exists) $i\geq 2$ such that
$v_i(h - a_1 q_1) = v_i(h)$  and the element $h - a_1 q_1$, and iteratively we find a linear combination
of monomials $p=\sum a_i q_i$ satisfying $\vv(h-p)\ge \vv(h)+\ee$.
Now, if $\mathcal E\neq\emptyset$ choose any  $\rho\in {\cal E}$ and set $Q=Q_{\rho}$ , and if $\mathcal E=\emptyset$ (in particular $r\ge 2$) choose generic $\lambda_1, \ldots, \lambda_r$ in $K^*$ and set $Q=\lambda_1f_1+\ldots +\lambda_rf_r$.
Repeating the above procedure the times we need,  we can finally obtain a finite
linear combination $q$ of monomials in ${\cal M}(\Lambda_{\k{\cal
E}})$ such that $\vv(h-q)\ge \delta + \vv({Q})$ where
$\delta$ is the conductor of the semigroup $S_C$. The element $g=(h-q)/Q$
of the total ring of fractions of ${\cal O}$ has value $\vv(g)\ge \delta$, in particular it belongs to
the integral closure
$\k{\cal O}$ of the ring ${\cal O}$ in its  total
ring of fractions (since
$\k{\cal O}$ is in fact the set of elements $\phi$ of the total
ring of fractions such that $v_i(\phi)\geq 0$ for all
$i=1,\ldots,r$).
Moreover, the
conductor ideal of $\k{\cal O}$ in ${\cal O}$ coincides with the
valuation ideal $J^C(\delta)$, so  $g\in {\cal O}$ and
then $h = q + g Q$ belongs to the ideal generated by
${\cal M}_{\vv(h)}(\Lambda_{\k{\cal E}})$. Thus, the set $\Lambda_{\k{\cal
E}}$ is a finite generating sequence for the plane curve $C$.

Now, we prove the theorem for the set $V=\{\nu_1,\ldots, \nu_r\}$.
The case $r=1$ is proved in \cite{s-2}, hence, by Theorem \ref{belga}, it suffices to
show that for any $h\in R$, one can find a linear combination of monomials $q$ in
$\Lambda_{\cal E}$ such that $\nu_i(h-q)>\nu_i(h)$ for all $i=1,\ldots,r$. Proposition
\ref{truemain}, applied recursively for $i=1,\ldots, r$, gives a finite sequence of
polynomials $\seq{q}1r$ in $\Lambda_{\cal E}$ such that $ \nu_{j}(h-\suma_{k=1}^i q_k)
>\nu_j(h)$ for $j \leq i$ and $\nu_{j}(h-\suma_{k=1}^i q_k) \ge \nu_j(h)$ for $j=1,2,
\ldots, r$. Hence, $q= \suma_{k=1}^r q_k$ satisfies our
requirements.

To prove the minimality, it is enough to check that any generating
sequence must have an element of type $Q_{\rho}$ (that is,
irreducible and with strict transform smooth and transversal to
$E_{\rho}$ at a nonsingular point) for each $\rho\in\mathcal E$.

Suppose $r=1$, and consider the minimal set of generators of the
semigroup $S_C$ or $S_V$, $\seq{\betab}0g$ (corresponding to
$\mathcal E=\{\rho_0,\ldots,\rho_g\}$). In order to generate
$J(\betab_i)$ ($0\le i\le g$), we need at least an element $h\in
R$ such that $v_1(h) = \nu_1(h)=\betab_i$. But it is known (see e.g.
\cite{man}) that in this case $h$ must be of type $Q_{\rho_i}$. In
the divisorial case, if $\alpha(1)$ is a dead end, moreover we have  to
consider $\betab_{g+1}=\nu_1(Q_{\alpha(1)})$ and from
Proposition~\ref{isom} we deduce that to generate
$J(\betab_{g+1})$ we need some element of type $Q_{\alpha(1)}$,
since all the elements $h\in R$ such that $\nu_1(h)=\betab_{g+1}$
and $\w{C}_h\cap E_{\alpha(1)}=\emptyset$ have the same initial
form.

Now, assume $r>1$. Notice that for any $W\subset V$ and $\mm'\in
S_W$, $J(\mm')=J(\min pr_W^{-1}(\mm'))$, where $pr_W:S_V\to S_W$
is the projection map. Thus, any generating sequence for $V$ is
also a generating sequence for $W$ and in particular for $\nu_i$,
$1\le i\le r$. On the other hand, if $\rho\in {\cal G}(\pi)$ is a
dead end of the minimal resolution $\pi$ of $V$, then there exists
$i\in \{1,\ldots, r\}$ such that $\rho$ is a dead end of the
minimal resolution of $\nu_i$. Therefore we cannot delete any
$Q_\rho$ in our generating sequence and a similar argument holds
for curves. Finally, in this last case, if $v_j(h)=v_j(f_i)$,
$j\neq i$ and $h \neq f_i$, then $v_i(h)<k$ for some positive
integer $k$ and $\mm=(v_1(f_i), \ldots,k, \ldots, v_r(f_i))$,
where $k$ is in the $i$th coordinate, belongs to $S_C$; hence no
 $f_i$ can be omitted to generate $J^C(\mm)$.
\end{proof}

\begin{remark} We have also proved that minimal generating
sequences for $V$ and $C$ must be of the form given in Theorem
\ref{maint}. On the other hand, Theorem \ref{belga} is also true
for the case of curves, so $[\Lambda_{\cal
\k{E}}]=\{[Q_{\rho}]\}\cup\{[f_1],\ldots, [f_r]\}$ is a set of
generators of ${\rm gr} {\cal O}$. However, here the set
$[\Lambda_{\cal \k{E}}]$ has {\bf infinitely many elements},
because $[f_i]_{\nn} \neq 0$ for infinitely many elements $\nn \in
S_{C}$, since $v_i(f_i)=\infty$.
\end{remark}

\section{Poincar\'e series}\label{poincar}

Along this section we will suppose $r>1$. Let ${\cal L}:=\ZZ[[t_1,
t_1^{-1},\ldots, t_r,t_r^{-1}]]$ be the set of formal Laurent
series in $t_1, \ldots, t_r$  and $\t^{\mm}:=t_1^{m_1}\cdot\ldots
\cdot t_r^{m_r}$ for $\mm=(m_1,\ldots, m_r)\in\mathbb Z^r$. ${\cal
L}$ is not a ring, but it is a $\ZZ[t_1, \ldots, t_r]$--module and
a $\ZZ[t_1, t_1^{-1}, \ldots, t_r, t_r^{-1}]$--module. For a
reduced plane curve $C$ with $r$ branches, the formal Laurent
series
$$ L_C(t_1,
\ldots,t_r)=\suma_{\mm\in\ZZ^r} c(\mm)\cdot\t^\mm\ \in\ {\cal L}\;
$$
was introduced in \cite{c-d-g}, where it was shown that
$$
P'_C(t_1, \cdots ,t_r)= L_C(t_1, \cdots, t_r)\cdot \pr_{i=1}^r
(t_i-1)
$$
is in fact a polynomial that is divisible by $t_1\cdots t_r -1$.
The Poincar\'e series for the curve $C$ was defined as the
polynomial with integer coefficients
$$ P_C (\seq t1r) = \dfrac{P'_C(\seq t1r)}{t_1\cdots
t_r -1}.
$$

Analogously, for a set of divisorial valuations $V=\{\seq
{\nu}1r\}$, we define
$$ L_V(t_1,
\ldots,t_r)=\sum\limits_{\mm\in\ZZ^r}d(\mm)\cdot\t^\mm\ \in\ {\cal L}. $$ $L_V$ is a
Laurent series, but, since $d(\mm)$ can be positive even if $\mm$ have some negative
component $m_i$, it is not a power series (in fact, as in the case of a curve, it
contains infinitely many terms with negative powers). In Proposition~\ref{prop_series},
we will show that
$$P'_V(t_1, \ldots ,t_r)= L_V(t_1, \ldots, t_r)\cdot \pr_{i=1}^r
(t_i-1) \in \ZZ\left[\left[\seq t1r \right]\right]\; . $$
 Thus, we
define the {\bf Poincar\'e series} of $V$ as the formal power
series with integer coefficients $$ P_V (\seq t1r) = \dfrac
{P'_{V}(\seq t1r)}{t_1\cdots t_r -1}\; . $$

\begin{remark}
It is not clear a priori whether the Laurent series $L_V$ can be computed from the Poincar\'e
series $P_V$, since in ${\cal L}$ there are elements which vanish after multiplication by
$\prod_{i=1}^r (t_i-1)$.  So, it is not obvious how to recover neither  the Hilbert
function of the graded ring $\gr_V R$,  $ d(\mm)$, $\mm\in \ZZ_{\ge 0}^r$, nor the
Hilbert function of the multi-index filtration of the ring
$R$,
$h(\mm) := \dim
R/J(\mm)$, $\mm\in \ZZ_{\ge 0}^r$. This can be done, following \cite{d-g},  as follows:
denote $I=\{1,\ldots,r\}$, and define
$$
\widetilde L_V(\seq t1r) = \sum\limits_{\mm\in\ZZ^r_{\ge 0}}
d(\mm)\cdot\t^{\mm} \in \ZZ[[\seq t1r]]\; ,
$$
and $ \widetilde P'_V(\seq t1r) = \widetilde L_V(\seq t1r)\cdot
\prod_{i=1}^r (t_i-1)$. The formula
$$
\widetilde P'_V(\seq t1r) = \sum_{J\subset I} (-1)^{\# J}\;
P'_V(\seq t1r) \mbox{\raisebox{-0.5ex}{$\vert$}}{}_{ {\{t_i=1
\mbox{ \small{for} } i\in J\}}}\; ,
$$
($\#J$ denoting the cardinality of $J$) allows to determine the series $\widetilde P'_V$
from the series $P'_V$, and as a consequence the power series $\widetilde L_V$. Finally,
$\widetilde L_V$ determines the Laurent series $L_V$, since $d(\mm) = d(\max
(m_1,0),\ldots,\max (m_r,0))$ for $\mm\not\le -\1$ and  $d(\mm)=0$\, for $\mm \le -\1$.

To get $h$, set $H(\seq t1r) := \sum\limits_{\mm\in \ZZ^r} h(\mm)\cdot \t^{\mm}\in {\cal
L}$, where $h(\mm)=\dim R/J(\mm)$ for $\mm\in \ZZ^r$ (notice that $h(\mm)=0$ if $\mm \leq 0$).  The equality $H(\seq t1r) = L_V(\seq t1r) (1 +
\t^{(1, \ldots,1)}+ \t^{(2, \ldots, 2)} + \ldots )$ solves our problem. Note that the
right hand side of the last equality makes sense since $d(\mm)=0$ for $\mm\le
-\underline{1}$.
\end{remark}

\bigskip

The following results involve dimensionality of the homogeneous
components of the graded algebra relative to finite sets
$V=\{\nu_1,\ldots, \nu_r\}$ of divisorial valuations. They will be
useful to relate the Poincar\'e series  of $V$ and the one of any
general curve for $V$. For $i \in I= \{1,2, \ldots, r\}$,  define
\begin{equation}\label{eq0}
\begin{array}{rl}
p_i(\mm) &:= \suma_{J \subset I\setminus\{i\}} (-1)^{\# J} d_i(\mm
+ \ee_J)\quad \mbox{and} \\
P_i(\seq t1r) &:= \suma_{\mm\in \ZZ^r} p_i(\mm)\t^{\mm}\in {\cal
L}\; .
\end{array}
\end{equation}

\begin{proposition}\label{prop_series}
Let $V$ be a finite set of $r$ divisorial valuations, then
$$
P'_V(t_1, \ldots ,t_r)= (t_1 t_2 \cdots t_r -1) P_i(\seq t1r) \in
\ZZ\left[\left[\seq t1r \right]\right]\; .
$$
As a consequence $P_V(\seq t1r) = P_i(\seq t1r)$ does not depend
on the index $i$ chosen. Moreover, if we write $P_V(\seq t1r) =
\suma_{\mm\in \ZZ^r} p(\mm)\t^{\mm}$, then
 $\mm\in S_V$ whenever $p(\mm) \neq 0$ and so $P_V(\seq t1r)\in \ZZ[[\seq t1r]]$.
\end{proposition}

\begin{proof}
We shall show the first statement for $i=1$ for the sake of
simplicity. Write  $P'_V(\seq t1r) = \suma_{\mm\in \ZZ^r}
\ell(\mm)\t^\mm \in {\cal L}$. Then
$$
\begin{array}{rl}
\ell(\mm) & = \suma_{J\subset I} (-1)^{\# J} d(\mm-\ee+\ee_J) \\
& = \suma_{J\subset I\setminus\{1\}} (-1)^{\# J} \left( d(\mm-\ee+\ee_J)-
d(\mm-\ee+\ee_J+\ee_1) \right)\; .
\end{array}
$$

On the other hand, if $\nn\in \ZZ^r$, then, for any arrangement
$(\seq i1r)$ of the elements in the set $I$, $ D(\nn)\simeq
\bigoplus_{j=1}^{r} D_{i_j}(\nn+ \ee_{i_1}+\cdots+\ee_{i_{j-1}})\;
$. So, $d(\nn) = \suma_{j=1}^r
d_{i_j}(\nn+\ee_{i_1}+\cdots+\ee_{i_{j-1}})$.

Applying the above decomposition for $\nn = \mm-\ee+\ee_J$ with
the natural arrangement $(1,2,\ldots, r)$ and for $\nn =
\mm-\ee+\ee_J+\ee_1$ with the arrangement $(2,3,\ldots, r,1)$ we
get:
$$
\begin{array}{rl}
\ell(\mm) & = \suma_{J\subset I\setminus\{1\}} (-1)^{\# J} \left(
d(\mm-\ee+\ee_J)- d(\mm-\ee+\ee_J+\ee_1) \right) \\
&= \suma_{J\subset I\setminus\{1\}}
(-1)^{\# J} \left( d_1(\mm-\ee+\ee_J)- d_1(\mm+\ee_J) \right) \\
& = p_1(\mm-\ee) - p_1(\mm)\; ,
\end{array}
$$
and thus, we obtain the formula for $P'_V$ given in the statement.
\smallskip

Now, let $\mm\in \ZZ^r$ be such that $\mm\notin S_V$. Then there
exists an index $i\in I$ such that $d_i(\mm)=0$, since otherwise
$$\mm = \min \{\nuv(g_i) | \init_{\nu_i}(g_i)\in D_i(\mm)\setminus
\{0\}, 1 \leq i \leq r\}$$ is in $S_V$. $D_i(\mm+\ee_J)\subset
D_i(\mm)$ for any $J\subset I$ with $i\notin J$, thus
$d_i(\mm+\ee_J)=0$ and so $p_i(\mm) = 0$, which ends the proof.
\end{proof}

We will say that the divisorial valuation $\nu_j\in V$ is {\bf
extremal} if $\alpha(j)$
is a dead end of ${\cal G}(\pi)$, $\pi$ being the minimal resolution of $V$.

\begin{lemma}\label{lemma_ter}
Take $\mm \in S_V$ and let $\nu_j\in V$ be an extremal valuation.
Then, for every $J\subset I$ with $j\notin J$ the equality
$d_j(\mm+B^j+\ee_J) = d_j(\mm+\ee_J)+ 1$ holds. As  a consequence,
$p(\mm) = p(\mm+B^j)$.
\end{lemma}

\begin{Proof}
Set $I'=I\setminus \{j\}$ and $B^j= (\seq {B^j}1r)$. Since $\nu_j$ is extremal, by
Proposition~\ref{mainl} there exists a monomial $q$ such that $\nu_j(q) =  B^j_j$ and
$\nu_i(q) > B^j_i$ for $i\in I'$. Let $h\in R$ be such that $\nuv(h) = \mm$; then $\nuv(h
q)\ge \mm+\ee_{I'}+B^j$ but $\nuv(h q)\not \ge \mm+\ee+B^j$. In particular, for any
$J\subset I'$, $\nuv(h q)\ge \mm+\ee_{J}+B^j$ but $\nuv(h q)\not \ge
\mm+\ee_{J}+\ee_j+B^j$ and so, $d_j(\mm+B^j+\ee_J) \neq 0$.

Consider again $J\subset I$ such that  $j\notin J$. If
$d_j(\mm+\ee_J) \neq 0$ then, by Lemma~\ref{lem6},
$d_j(\mm+B^j+\ee_J)=d_j(\mm+\ee_J)+1$. If, otherwise,
$d_j(\mm+\ee_J)=0$,  then $d_j(\mm+\ee_J+B^j) =1$, since
$d_j(\mm+\ee_J+B^j)  \ge  2$ implies $d_j(\mm+\ee_J)  \ge  1$
(Lemma \ref{lem5}).

Finally, the fact that $p(\nn) = p_j(\nn)$ for any $\nn\in\mathbb Z^r$ and the following
equalities chain conclude the proof (recall that $r>1$)
$$
\begin{array}{rl}
p_j(\mm+B^j) & = \suma_{j\notin J\subset I} (-1)^{\# J}
d_j(\mm+B^j+\ee_J)
\\
& = \suma_{j\notin J\subset I} (-1)^{\# J} \left( d_j(\mm+\ee_J)
+1\right)
\\
& = p_j(\mm) + \suma_{j\notin J\subset I}(-1)^{\# J} = p(\mm). \; \; \Box
\end{array}
$$
\end{Proof}

Let $C=C_f$ be a general curve of $V$ and consider the sequence of families of valuations
$\{V^{(k)}\}$ constructed for $C$ in Section \ref{gradedalg}, where $\pi^{(0)}$ is the
minimal resolution of $V$. Next proposition allows to write the Poincar\'e series,
$P_{V^{(k)}}(\seq t1r)$, as a quotient of two series in such a way that the numerator
does not depend on $k$. Stand $B_{(k)}^i$ for the value $B^i$ associated to the family
$V^{(k)}$.
\begin{proposition}\label{lem8} For every $k\ge 0$,
$$
P_{V^{(k)}}(\seq t1r)\cdot \pr_{i=1}^r (1-\t^{B_{(k)}^i}) =
P_{V}(\seq t1r) \cdot \pr_{i=1}^r (1-\t^{B^i})\; .
$$
\end{proposition}

\begin{proof}
It suffices to show the formula for $k=1$. Let $E_{\w \alpha(1)}$ be
the exceptional divisor created by blowing-up at a smooth point
$P\in E_{\alpha(1)}$. Set $\w \nu_1$  the $E_{\w{\alpha}(1)}$
valuation and consider the set of
divisorial valuations $\w V = \{\w \nu_1, \nu_2,\ldots, \nu_r\}$.
Stand $P_{\w V}(\seq t1r)$ for the Poincar\'e series of $\w V$  and
set $\w B^1 = \left(\w \nu_1(Q_{\w \alpha(1)}), \nu_2(Q_{\w
\alpha(1)}),\ldots, \nu_r(Q_{\w \alpha(1)}) \right)\in S_{\w V}$. If
we prove that
$$
(1-\t^{\w B^1}) P_{\w V}(\seq t1r) = (1-\t^{B^1}) P_V(\seq t1r),
$$
then the result, for $k=1$, follows after iterating the same
procedure for the remaining valuations $\nu_i$.

Let us write $P_{\w V}(\seq t1r) = \suma_{\mm\in \ZZ^r}
\w{p}(\mm)\t^{\mm}$. We only need to prove, for any $\mm\in
\ZZ^r$, the following equality:
\begin{equation}\label{eq1}
\w{p}(\mm)- \w{p}(\mm-\w{B}^1) = p(\mm)- p(\mm-B^1).
\end{equation}
Indeed, if $\mm\notin S_{V}$ (respectively, $\mm\notin S_{\w V}$)
then the right (respectively, the left) hand side of
equality~(\ref{eq1}) vanishes, since both  involved terms are equal
to zero. Moreover, if $\mm\in S_{\w V}\setminus S_V$ then by
Theorem~\ref{str_semi}, $\mm  = \nn + s \w B^1$ for some $\nn\in S_{V}$ and $s\ge 1$ (since otherwise $\mm\in S_V$). In particular, $\mm- \w B^1\in S_{\w V}$ and,
since $\w \nu_1$ is extremal, Lemma~\ref{lemma_ter} implies
$\w{p}(\mm) = \w{p}(\mm-\w{B}^1)$. Therefore the left hand side of
equality~(\ref{eq1}) is also equal to zero. So, from now on we
assume that $\mm\in S_{V}$.

Denote by $\w J(\mm)$ the valuation ideal of $\mm$ for $\w V$. Set
$\w d_1(\mm) = \dim \w J(\mm)/$ $\w J(\mm+\ee_1)$. Taking into
account the formulae in (\ref{eq0}), to prove (\ref{eq1}) we only
need to show the following equality for any $J\subset
I\setminus\{1\}$:
\begin{equation}\label{eq2}
\w d_1(\mm+\ee_J) - \w d_1(\mm-\w B^1+\ee_J) = d_1(\mm+\ee_J) -
d_1(\mm-B^1+\ee_J)\; .
\end{equation}

Let us assume that either $d_1(\mm + \ee_J)\neq 0$ or $\w{d}_1(\mm
+ \ee_J)=0$. Since $S_V\subset S_{\w{V}}$, we have $d_1(\nn)=0$ if
$\w{d}_1(\nn)= 0$, for any $\nn\in \mathbb Z^r$. Therefore, if
$\w{d}_1(\mm-B^1+e_J)=0$, then $d_1(\mm-B^1+e_J)=0$, and by Lemma
\ref{lem5}, $\w{d}_1(\mm+e_J)\in\{0,1\}$ and
$d_1(\mm+e_J)\in\{0,1\}$. Since our assumption excludes the case
$\w{d}_1(\mm+e_J)=1, \; d_1(\mm+e_J)=0$, the equality (\ref{eq2})
holds. Otherwise, $\w{d}_1(\mm-B^1+e_J)\neq 0$, by Lemma
\ref{lem6} the left hand side of the equality (\ref{eq2}) is equal
to $1$. Again by our assumption, we cannot have $d_1(\mm+e_J)=0$
and applying Lemma \ref{lem5} when $d_1(\mm-B^1+e_J)=0$ and Lemma
\ref{lem6} otherwise we prove that the right hand side  also
equals $1$.

To finish the proof, we will prove that there is no  $J\subset
I\setminus\{1\}$ such that $d_1(\mm+\ee_J)=0$ and $\w
d_1(\mm+\ee_J) \neq 0$. If $d_1(\mm+\ee_J)=0$ and $\w
d_1(\mm+\ee_J) \neq 0$,  pick $h\in R$ such that its image in
$\w{D}_1(\mm+\ee_J)$ does not vanish. Let  $\w \pi$ be the minimal
resolution of $\w V$. Clearly, $\mathcal G(\w{\pi}) \setminus \{\w
\alpha(1)\}$ is connected. Since $d_1(\mm+\ee_J)=0$, we have $\w
\nu_1(h) \neq \nu_1(h)$, and as a consequence the strict transform
of $C_h$ by $\w \pi$ intersects $E_{\w \alpha(1)}$. Let $h =
\varphi h'$ be such that the strict transform of $C_{h'}$ by $\w
\pi$ does not intersect $E_{\w \alpha(1)}$ and the strict
transforms by $\w \pi$ of all the irreducible components of
$C_\varphi$ intersect $E_{\w \alpha(1)}$.

Applying Proposition~\ref{mainl} to the connected component $\Delta = \mathcal
G(\w{\pi})\setminus\{\w\alpha(1)\}$ one can show that
there exists a monomial $q$ in the elements $\{Q_\rho\mid\rho\in
{\cal E}\cap \Delta\}$, such
that $\w \nu_1(q) = \w \nu_1(\varphi)$ and $\nu_i (q) > \nu_i(\varphi)$ for
$i=2,\ldots,r$. As the irreducible components of the strict transforms of $C_q$ do not
meet the divisor $E_{\w \alpha(1)}$, one has $\nu_1(q) = \w \nu_1(q)$ and so $\nu_1(h' q)
=\w \nu_1(h' q) = \w \nu_1(h)=m_1$ and $\nu_i(h' q)
> \nu_i (h)\ge m_i$. As a consequence, $h'q \in D_1(\mm +
\ee_{I \setminus \{1\}})\setminus\{0\}$, and then $d_1(\mm+\ee_J) \neq 0$, which is a
contradiction.
\end{proof}

Now, we state the relationship between the Poincar\'e series of a finite set of
divisorial valuations $V$ and the Poincar\'e polynomial of a general curve, $C$, of $V$.

\begin{theorem}\label{poincare1}
Let $V$ and $C$ be as above. Then,
$$
P_V(\seq t1r) = \dfrac{P_C(\seq t1r)}{\pr_{i=1}^r (1- \t^{B^i})}\;
.
$$
\end{theorem}

\begin{proof}
By Proposition~\ref{lem8}, it suffices to prove the result for the set of valuations
$V^{(k)}$,  for some $k$. In particular, we can assume that all the divisorial valuations
$\seq{\nu}1r$ are extremal. Fix some $k$ and, for simplicity, write $\w{V}=V^{(k)}$,
$P_{\w{V}}(\seq t1r) = \suma_{\mm\in \ZZ_{\ge 0}^r} \w{p}(\mm)\t^\mm$, and set $\w{d_i}$
for the corresponding dimensions, and recall that $\w{p}(\mm) = 0$ when $\mm\notin
S_{\w{V}}$.

The coefficient of $\t^{\mm}$ in the series $P_{\w{V}}(\seq t1r) \cdot \pr_{i=1}^r
(1-\t^{\w{B^i}})$ is
$$
\lambda_{\mm}=\suma_{J\subset I}(-1)^{\#J}\w{p}(\mm-\suma_{i\in J} \w{B^i}).
$$

Now, if $J_{\mm}=\{i\in I\mid \mm-\w{B^i}\in S_{\w{V}}\}$, we have $\mm-\suma_{i\in
J}\w{B^i}\in S_{\w{V}}$ if and only if $J\subset J_{\mm}$ (see Theorem \ref{str_semi}),
and in this case, by Lemma \ref{lemma_ter}, $\w{p}(\mm-\suma_{i\in
J}\w{B^i})=\w{p}(\mm)$. Therefore,
$$
\lambda_{\mm}=\suma_{J\subset J_{\mm}}(-1)^{\#J}\w{p}(\mm-\suma_{i\in J}
\w{B^i})=\suma_{J\subset J_{\mm}}(-1)^{\#J}\w{p}(\mm),
$$
which is $0$ if $J_\mm\neq \emptyset$ and $\w{p}(\mm)$ if $J_{\mm}=\emptyset$, that is,
if $\w{d_i}(\mm)=1$ for $1 \leq i \leq r$ (Theorem \ref{str_semi}). Hence,
$$
P_{\w{V}}(\seq t1r) \cdot \pr_{i=1}^r (1-\t^{\w{B^i}})= \suma_{m\in A}
\w{p}(\mm)\t^{\mm},
$$
where
$$
A := \{ \mm\in S_{\w{V}} \mid \w{d_i}(\mm)=1 \text{ for } 1 \leq i
\leq r \}\; .
$$

For $J\subset I\setminus\{1\}$ we have $D_1(\mm+e_J)\subset D_1(\mm)$, hence
$\w{d_1}(\mm+e_J)\le 1$ for any $\mm\in A$. Thus, all the summands in the formula
$\w{p}(\mm) = \w{p}_1(\mm)=\suma_{J\subset I\setminus\{1\}}(-1)^{\# J}
\w{d_1}(\mm+\ee_J)$ are $1$ or $0$.

\smallskip

On the other hand, if $P_C(\seq t1r)=\suma_{\mm}  \bar{p}(\mm)\t^\mm$ is the Poincar\'e
polynomial of the curve $C$, it is straightforward to deduce that the coefficients
$\bar{p}(\mm)$  satisfy a formula similar to (\ref{eq0}), in particular the following
equality holds
$$
\bar{p}(\mm) = \bar{p}_1(\mm) = \suma_{J\subset I\setminus\{1\}}(-1)^{\# J}
c_1(\mm+\ee_J)\;
$$
where $c_1(\nn) = \dim J^C(\nn)/J^C(\nn+e_1)$ for any
$\nn\in\ZZ^r$. And in this case, the dimensions $c_1(\nn)$ only
 could be $1$ or $0$, because $ J^C(\nn)/J^C(\nn+e_1)$ can be regarded as a vector subspace of  $J^{C_1}(\nn)/J^{C_1}(\nn+e_1)$ ($C_1$ being one of the branches of $C$), whose dimension is $1$ or $0$.

We claim that there exists some $k$ such that for any $\mm\in A$
and $J\subset I\setminus\{1\}$ it happens that
$\w{d_1}(\mm+\ee_J)=0$ if and only if $c_1(\mm+\ee_J)=0$, and such
that $\bar{p}(\mm)=0$ for any $\mm\notin  A$. Then, we deduce
$P_{\w{V}}(\seq t1r) = \dfrac{P_C(\seq t1r)}{\pr_{i=1}^r (1-
\t^{B^i})}$, as we wanted to prove (recall that $\w{V}$ depends on
$k$).

Since $S_{\w{V}}\subset S_C$, then $\w{d_1}(\nn)\neq 0$  implies
$c_1(\nn)\neq 0$ for any $\nn\in \ZZ^r$. Moreover, for $\mm$
fixed, there exists $k_0$ such that, for any $J\subset
I\setminus\{1\}$ and $k\ge k_0$, we have $d^{(k)}_1(\mm+\ee_J)
\neq 0$ if $c_1(\mm+\ee_J) \neq 0$. But $P_C$ is a polynomial, so
$B:=\{\mm\in S_C\mid \bar{p}(\mm)\neq 0\}$ is a finite set and for
$k$ large enough we find $\w{d_1}(\mm+\ee_J)=0$ if and only if
$c_1(\mm+\ee_J)=0$, for any $J\subset I\setminus\{1\}$ and $\mm\in
B$; if moreover we pick $k$ such that $\mm\not\geq B^i_{(k)}$ for
every $\mm\in B$, we have $B\subset A$ (see Theorem
\ref{str_semi}), that is, $\bar{p}(\mm)=0$ for any $\mm\notin  A$,
which proves our claim.
\end{proof}

\medskip

Next corollary gives a precise meaning to the fact that the
valuations defined by  a curve singularity can be approached by
families of divisorial valuations:

\begin{corollary}\label{cor} Let $V$ and $C$ as above and
$V^{(k)}$ ($k\ge 0$)  the finite sets of divisorial valuations
defined in Section \ref{gradedalg}.  Then,
$$
\lim_{k\to \infty} P_{V^{(k)}}(\seq t1r) = P_C(\seq t1r).\; \; \Box
$$
\end{corollary}

Finally,  assume that $R = {\cal O}_{\C^2,0}$ is the local ring of
germs of holomorphic functions at the origin of the complex plane
$\C^2$. For a vertex $\alpha$ of the dual graph ${\cal G}$ of a
set of valuations $V$ as above, denote by
${\stackrel{\bullet}{E}}_\alpha= E_\alpha \setminus
(\k{E-E_\alpha})$ the smooth part of an irreducible component
$E_{\alpha}$ in the exceptional divisor, $E$, of the minimal
resolution of $V$ and by $\chi({\stackrel{\bullet}{E}}_\alpha)$
its Euler characteristic. In addition, set $\nuv^\alpha =
\nuv(Q_{\alpha})$. Then  the following formula of A'Campo's type
\cite{acam}, firstly proved in \cite{d-g},  holds.

\begin{corollary}\label{poincare2}
$$
P_V(\seq t1r) = \pr\limits_{E_\alpha\subset{E}}
\left(1-\t^{{\nuv}^\alpha}\right)^{-\chi({\stackrel{\bullet}{E}}_\alpha)}\;
.
$$
\end{corollary}

\begin{Proof}
$E_{\alpha}$ is isomorphic to the complex line $\PP^1_\C$, so
$\chi({\stackrel{\bullet}{E}}_\alpha)= 2-b(\alpha)$, where
$b(\alpha)$ denotes the number of singular points of $E_{\alpha}$ in
$E$ (i.e., the number of connected components of ${\cal G} \setminus
\{\alpha\}$).

Since the Poincar\'e
polynomial $P_C(\seq t1r)$ coincides with the Alexander polynomial
$\Delta^{C}(\seq t1r)$ (see \cite{c-d-g}) and by using the
Eisenbud-Neumann
 formula
for $\Delta^C(\seq t1r)$ \cite{EN}, we obtain:
$$
P_C(\seq t1r) = \Delta^C(\seq t1r) =
\prod\limits_{E_\alpha\subset{E}}
\left(1-\t^{{\vv}^\alpha}\right)^{-\chi({\stackrel{\circ}{E}}_\alpha)}\;
,
$$
where $\vv^\alpha = \vv(Q_{\alpha})$, $\vv$ being the above
described valuation  sequence given by $C$,  and
${\stackrel{\circ}{E}}_\alpha$ is the smooth part of $E_\alpha$ in
the total transform of $C$ by the minimal resolution of $V$.
$\chi({\stackrel{\circ}{E}}_\alpha)= 2-b(\alpha)=
\chi({\stackrel{\bullet}{E}}_\alpha)$ for those $\alpha\notin
\{\alpha(1), \ldots, \alpha(r)\}$ and
$\chi({\stackrel{\circ}{E}}_{\alpha(i)})= 2-(b(\alpha(i))+1)=
\chi({\stackrel{\bullet}{E}}_{\alpha(i)})-1$ for $i=1,\ldots,r$.

Finally,  $\vv^{\alpha} = \nuv^{\alpha}$ for $\alpha\in {\cal G}$
and, so,
$$
P_V(\seq t1r) = \frac{\prod\limits_{E_\alpha\subset{E}}
\left(1-\t^{{\vv}^\alpha}\right)^{-\chi({\stackrel{\circ}{E}}_\alpha)}}
{\pr_{i=1}^r (1-\t^{B^i})} = \prod\limits_{E_\alpha\subset{E}}
\left(1-\t^{{\vv}^\alpha}\right)^{-\chi({\stackrel{\bullet}{E}}_\alpha)}.
\; \;\Box
$$
\end{Proof}

\begin{remark}
The proof of the equality between the Poincar\'e and the Alexander polynomials in \cite{c-d-g} uses the topology of the complex field. However, the authors have informed us about the existence of a non-published alternative proof
which avoids this, by using deeper properties of the semigroup $S_C$. Thus,  writing $2-b(\alpha)$ instead $\chi({\stackrel{\bullet}{E}}_\alpha)$, the above formula holds also in the non-complex case.
\end{remark}

\bigskip
 We conclude this paper giving an illustrative
example.

\medskip

 \noindent
{\bf Example.} Let $x,y$ be independent variables and set $T= \C[x,y]_{(x,y)}$. Consider
the set $V= \{\nu_1, \nu_2, \nu_3 \}$ of divisorial valuations of $\C (x,y)$ centered at
$T$, whose minimal resolution is given by the sequence of ideals:
    \begin{itemize}
    \item $\nu_1$: $\ideal m_0=\langle x,y \rangle, \ideal m_1=\langle x,\dfrac{y}{x}\rangle, \ideal m_2=\langle
    x,\dfrac{y}{x^2}-1\rangle$;
    \item $\nu_2$: $\ideal m_0, \ideal m_1, \ideal m_3=\langle
    \dfrac{y}{x},\dfrac{x^2}{y}\rangle, \ideal m_4=\langle \dfrac{y}{x},\dfrac{x^3-y^2}{y^2}\rangle,
    \ideal m_5=\langle \dfrac{x^3-y^2}{y^2},
    \dfrac{y^3}{x(x^3-y^2)}\rangle$,\\
    $\ideal m_6=\langle \dfrac{x^3-y^2}{y^2}, \dfrac{y^5}{x(x^3-y^2)^2}
    -1\rangle$;
    \item $\nu_3$: $\ideal m_0, \ideal m_1, \ideal m_3, \ideal m_4, \ideal m_5$.
    \end{itemize}

\begin{figure}[h]
\unitlength=1.00mm
    \begin{picture}(60.00,21.00)(-20,10)
    \thicklines \put(-5,30){\line(1,0){90}} \put(-5,30){\circle*{1}}
    \put(25,30){\circle*{1}} \put(55,30){\circle*{1}}
    \put(85,30){\circle{2}} \put(85,30){\circle*{1}}
    \put(55,30){\circle{2}} \put(-6,25){{\bf 1}}
    \put(22,32){$4$} \put(47,33){$\alpha(\nu_3) = 6$}
    \put(80,33){$\alpha(\nu_2) = 7$} \put(55,15){\line(0,1){15}}
    \thicklines \put(25,15){\line(0,1){15}} \put(25,15){\circle*{1}}
    \put(55,15){\circle*{1}} \put(40,15){\circle*{1}}
    \put(40,15){\circle{2}} \put(22,15){$2$}
    \put(31,10){$\alpha(\nu_1) = 3$} \put(57,13){$5$}
    \put(25,15){\line(1,0){15}}
    \end{picture}
\caption{Dual graph} \label{fig2}
\end{figure}
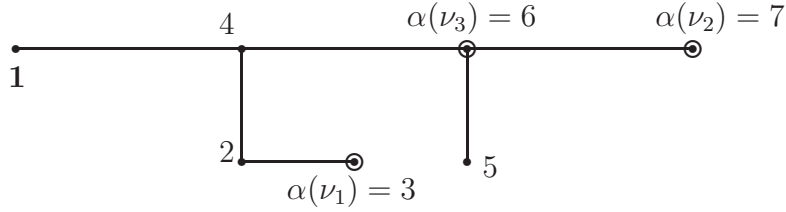

The dual graph of $V$ has the shape of Figure \ref{fig2}. From it, we can deduce the
values $\nuv^{\bf 1}=(1,4,4)$, $\nuv^2=(2,6,6)$, $\nuv^3=(3,6,6)$, $\nuv^4=(3,12,12)$,
$\nuv^5=(3,13,13)$, $\nuv^6=(6,26,26)$, $\nuv^7=(6,27,26)$, as well as the values
$\chi({\stackrel{\bullet}{E}}_\alpha)=2-b(\alpha)$ (Corollary \ref{poincare2}), giving
the following expression for the Poincar\'e series
$$
P_V= \frac{(1-t_1^{3}t_2^{12}t_3^{12})(1-t_1^{6}t_2^{26}t_3^{26})}
{(1-t_1t_2^{4}t_3^{4})(1-t_1^3t_2^{6}t_3^{6})
(1-t_1^{3}t_2^{13}t_3^{13})(1-t_1^{6}t_2^{27}t_3^{26})}\;.
$$

Moreover, by Theorem \ref{maint}, the set $ \Lambda = \{ Q_{\bf
1}=x,Q_{3}= y-x^2, Q_5 = y^2 -x^3, Q_7 = (y^2 -x^3)^2 - x^5 y \}$
is a minimal generating sequence of $V$ since the set $\{ {\bf 1},
3, 5, 7\}$ is the set of dead ends of the displayed dual graph.

\vspace{3mm}
\par
\end{document}